\documentclass[11pt, a4paper]{article}

\usepackage[margin=2cm]{geometry}
\usepackage{xcolor}
\usepackage{booktabs}
\usepackage{mathtools}
\usepackage{url}
\PassOptionsToPackage{hyphens}{url}
\usepackage{hyperref}
\usepackage{amsthm}
\usepackage{tabularx}
\usepackage{multirow}
\usepackage{makecell}
\usepackage{float}
\usepackage{comment}
\usepackage{amsfonts}
\usepackage{siunitx}
\usepackage{cancel}
\usepackage{tikz}
\usepackage{subcaption}
\usepackage{gensymb}
\usetikzlibrary{arrows, shapes, positioning}
\usepackage{bbm}
\usepackage[numbers]{natbib}
\usepackage{enumerate}
\usepackage{lscape}
\usepackage{amssymb}
\usepackage{orcidlink}

\usepackage{algorithm}
\usepackage[noend]{algpseudocode}

\newtheorem{theorem}{Theorem}[section]
\newtheorem{corollary}{Corollary}[theorem]
\newtheorem{lemma}[theorem]{Lemma}
\newtheorem{proposition}[theorem]{Proposition}

\newtheorem{definition}[theorem]{Definition}

\newcommand{\tr}[1]{{\text{tr} #1}}
\newcommand{\diag}[1]{{\text{diag}\left(#1\right)}}
\newcommand{\Diag}[1]{{\text{Diag}\left(#1\right)}}
\newcommand{\rank}[1]{{\text{rank}\left(#1\right)}}

\newcommand{\IP}[2]{\left\langle #1, #2 \right\rangle}
\newcommand{\eqdef}{\mathrel{\mathop:}=}
\newcommand{\psd}{\succeq}
\newcommand{\argmax}{\text{arg max}}
\newcommand{\argmin}{\text{arg min}}

\newcommand{\nsd}{\preceq}

\newcommand\norm[1]{\left\lVert#1\right\rVert}
\newcommand{\mc}{\mathcal}
\newcommand{\ceil}[1]{\left\lceil #1 \right\rceil}
\newcommand{\floor}[1]{\left\lfloor #1 \right\rfloor}
\newcommand{\aug}[1]{\hat{#1}}
\newcommand{\eps}{\varepsilon}

\begin{document}

\title{Improved semidefinite programming bounds for the maximum $k$-colorable subgraph problem}
\author{Mathijs Barkel\,\orcidlink{0000-0003-1276-588X}\thanks{Tilburg University, Department of Econometrics \& Operations Research, CentER, 5000 LE Tilburg, 
\href{mailto:m.j.barkel@tilburguniversity.edu}{m.j.barkel@tilburguniversity.edu} (corresponding author)}, 
\and
Renata Sotirov\,\orcidlink{0000-0002-3298-7255}\thanks{Tilburg University, Department of Econometrics \& Operations Research, CentER, 5000 LE Tilburg, 
\href{mailto:r.sotirov@tilburguniversity.edu}{r.sotirov@tilburguniversity.edu}}
}

\date{ }
\noindent
\maketitle

\begin{abstract}
We study the maximum $k$-colorable subgraph (M$k$CS) problem, which consists in finding  a largest $k$-colorable induced subgraph in a given graph.  
We consider a Semidefinite Programming (SDP) relaxation for the M$k$CS problem and regard its resulting upper bound as a graph parameter.
We present several properties of this graph parameter, from which we obtain that the M$k$CS problem is solvable in polynomial time for $k$-perfect graphs.
We further derive two novel families of valid inequalities to strengthen the SDP relaxation.
The first family reduces to a family of inequalities for the Boolean quadric polytope when $k=1$, and the second family generalizes the family of rank inequalities for binary linear programming formulations of the stable set problem.
We efficiently solve the strengthened SDP relaxation using a cutting-plane algorithm that is based on the Alternating Direction Method of Multipliers (ADMM). 
Extensive computational experiments show that the obtained upper bounds outperform the best upper bounds from the literature.
To complement our SDP-based upper bounds, we propose an integer ADMM variant that uses an exact Binary Semidefinite Programming (BSDP) formulation of the M$k$CS problem to produce high-quality feasible solutions. 
To the best of our knowledge, this is the first application of the ADMM to compute integer solutions to a BSDP problem.

\end{abstract}

\section*{Acknowledgments}
The authors thank Frank de Meijer for his tabu search code for the graph coloring problem, which is used to cluster the valid inequalities.
The second author would also like to thank Renate van der Knaap for her tabu search algorithm for the M$k$CS problem. 
This work used the Dutch national e-infrastructure with the support of the SURF Cooperative using grant no. EINF-13681.

\section{Introduction}
The maximum $k$-colorable subgraph (M$k$CS) problem is to find the largest induced subgraph in a given graph that can be colored in $k$ colors such that no two adjacent vertices have the same color.
Since a partial $k$-coloring of a graph is a set of pairwise disjoint stable sets in the graph, each one corresponding to a color, the solution for the M$k$CS problem corresponds to the largest possible number of colored vertices in a partial $k$-coloring.
The M$k$CS problem is also known as the maximum $k$-partite induced subgraph problem, as well as the maximum bipartite subgraph problem for the case $k=2$, while for  $k=1$ it coincides with the stable set problem.  
The M$k$CS problem is known to be $\mc{NP}$-hard~\citep{Lewis1980NPHardness}.

The M$k$CS problem arises in various practical applications, including channel assignment in spectrum-sharing networks
\citep{ApplicationChannelAssignment5, ApplicationChannelAssignment1, ApplicationChannelAssignment4, ApplicationChannelAssignment3, ApplicationChannelAssignment2}, very-large-scale integration design \citep{ApplicationVLSI2, ApplicationVLSI1}, human genetic research \citep{ApplicationVLSI2, ApplicationGenetics} and job scheduling and register allocation \citep{ApplicationJobScheduling}. 
The M$k$CS problem is also of interest due to its relation to several well-known combinatorial optimization problems. For instance, it is closely related to the graph coloring problem and the clique cover problem.

Several works address the M$k$CS problem using integer programming, some of which exploit that the M$k$CS problem is invariant under color permutations and graph automorphisms.
\citet{Campelo2010ILPLagr} introduce a Binary Linear Programming (BLP) formulation where color symmetry is handled through the selection of representatives of stable sets. 
They further propose a parallelized subgradient algorithm that uses a Lagrangian relaxation to decompose the M$k$CS problem into smaller weighted stable set problems. 
Lower and upper bounds are computed for instances with up to $500$ vertices and $3$ colors.
\citet{Januschowski2011ILPSym, Januschowski2011MkCS} propose an alternative BLP formulation as the basis of a branch-cut-and-propagate algorithm that leverages both valid inequalities and domain propagation to handle color and graph symmetry. 
They solve instances with up to $1085$ vertices and $25$ colors to optimality.
\citet{Quintero2022Quantum} derive two quadratic unconstrained binary optimization formulations for the M$k$CS problem that are solved using D-Wave’s quantum annealing device. 
The authors consider instances with up to $50$ vertices and $5$ colors.

Another line of research focuses on Semidefinite Programming (SDP).
\citet{Narasimhan1990GenTheta} propose an eigenvalue bound for the M$k$CS problem called the generalized $\vartheta$-number of a graph~$\mc{G}$ (denoted by $\vartheta_k (\mc{G})$), which was later reformulated by \citet{Alizadeh1995SDPGenTheta} as the solution of a semidefinite program. 
This graph parameter coincides with the celebrated $\vartheta$-number by \citet{Lovasz1979} for $k=1$, and provides also a lower bound on the minimum number of colors needed for a valid $k$-multicoloring of the complement of the considered graph.
\citet{Sinjorgo2022GenVartheta} study $\vartheta_k (\mc{G})$ for highly symmetric graphs and provide closed-form solutions  for several graph classes.
\citet{Kuryatnikova2022MkCSProblem} strengthen $\vartheta_k (\mc{G})$ by incorporating nonnegativity constraints, and consider several other SDP relaxations for the  M$k$CS problem.
Their strongest relaxation is derived via a vector-lifting approach, followed by a reduction in the size of the SDP relaxation that exploits the invariance of the M$k$CS problem under color permutations. 
This relaxation is further strengthened using inequalities from the Boolean quadric polytope (BQP).
A weaker but more scalable relaxation is obtained through a matrix-lifting approach.
The authors compare the strength of their relaxations from a theoretical point of view, and provide computational experiments on benchmark instances with up to $500$ vertices. 
Several instances are solved to optimality, and the strongest upper bounds improve upon the upper bounds from \cite{Campelo2010ILPLagr} for all but two instances.

Motivated by this success, we continue the line of research on SDP-based approaches for the M$k$CS problem. 
In particular, rather than using Interior Point Methods (IPMs), the classical approach for solving SDP problems, we consider the Alternating Direction Method of Multipliers (ADMM). 
This is a first-order method introduced in the 1970s to solve general convex optimization problems, see e.g., \cite{Boyd2011ADMMTheory}. 
Compared to IPMs, the time and memory requirements of the ADMM scale better with the order of the matrix variables and the number of constraints.
Indeed, the ADMM, and variants thereof, have been successfully applied to solve large-scale SDP relaxations of various combinatorial optimization problems, see e.g., \cite{Hu2020ADMMQSPP, Li2021PRSMMinCut, Oliveira2018ADMMForQAP, Sinjorgo2025ADMMStability}.

\subsection{Main results and outline}
We start our study of the M$k$CS problem with a Binary Semidefinite Programming (BSDP) formulation proposed by \citet{deMeijer2024BSDP}, from which we derive our basic SDP relaxation that coincides with the matrix-lifting relaxation from \cite{Kuryatnikova2022MkCSProblem}. 
We consider the optimal value of the basic SDP relaxation, denoted by $\theta_k (\mc{G})$, as a graph parameter.
We derive several properties of this graph parameter, including that  $\theta_k (\mc{G})$ is upper bounded by $k \vartheta(G)$ for all $k\geq 1$, and that the sequence $\left( \theta_k(\mc{G}) \right )_k$ is monotonically increasing in $k$. 
We also show that our graph parameter serves as  a lower bound for the $k$-clique cover number of $\mc{G}$.
\citet{lovasz1983perfect} defined a $k$-perfect graph as a graph for which the number of vertices in a maximum $k$-colorable subgraph equals the $k$-clique cover number for each induced subgraph of $\mc{G}$.
Therefore, these two graph parameters can be computed for $k$-perfect graphs in polynomial time up to fixed precision by solving our SDP relaxation for the M$k$CS problem.
\citet{lovasz1983perfect} posed an open question on the computability of these parameters for perfect graphs, and we provide an answer for $k$-perfect graphs where $k>1$. 
The case $k=1$ is a well-known result by \citet{Grtschel1984PolynomialAF}.
Moreover, we provide an answer to  Lov\'asz question for a class of perfect graphs known as totally perfect graphs. 

We also propose two novel families of valid inequalities to strengthen our SDP relaxation.
Our first family of inequalities reduces to a family of inequalities for the BQP by \citet{Padberg1989BQP} when $k=1$.
These new inequalities are not only relevant for our SDP relaxation for the M$k$CS problem, but potentially also for matrix-lifting relaxations of other packing or partitioning problems.
Our second family of inequalities generalizes the family of rank inequalities for BLP formulations of the stable set problem by \citet{Nemhauser1974PropertiesPolyhedra}. 
Given that there are exponentially many generalized BQP and rank inequalities, we also consider computationally tractable special cases of these inequalities that concern triangles, cliques and odd holes.

We solve the strengthened SDP relaxation using a recently proposed extension of the ADMM, called the cutting-plane ADMM (CP-ADMM), which integrates the ADMM with Dykstra's cyclic projection algorithm, see \cite{deMeijer2025ADMMQMSTP, deMeijer2021QCCP, deMeijer2023ADMMPartitioning}.
We improve the performance of the CP-ADMM by limiting the number of new cuts in which each variable may appear; a feature that was not yet used in existing CP-ADMM algorithms.
Extensive computational experiments on benchmark instances show that the strengthened SDP relaxation can be solved efficiently, and that the upper bounds obtained outperform the best upper bounds from the literature.

We  present a novel integer ADMM variant, denoted by INT-ADMM, that is aimed at constructing feasible solutions for the BSDP formulation, thereby yielding valid lower bounds.
Our INT-ADMM is based on the $\ell_p$-Box ADMM introduced by \citet{Wu2019LpBoxADMM}. To the best of our knowledge, we are the first to apply the key idea of the $\ell_p$-Box ADMM to solve a BSDP problem.
The INT-ADMM is not guaranteed to converge to an optimal or even a feasible solution, but our computational experiments show that our algorithm consistently leads to high-quality feasible solutions, often improving upon the best solutions in the literature. 
This success is partly due to an effective mechanism for escaping local optima.

This paper is further structured as follows.
In Section~\ref{Sec: BSDP formulation}, we formally introduce the M$k$CS problem, its BSDP formulation and our basic SDP relaxation.
In Section~\ref{Sec: Graph parameter}, we present our theoretical results on the graph parameter $\theta_k(\mc{G})$.
In Section~\ref{Sec: Valid inequalities}, we derive the valid inequalities.
In Sections~\ref{Sec: CP-ADMM} and~\ref{Sec: INT-ADMM}, we present the CP-ADMM and INT-ADMM, respectively.
In Section~\ref{Sec: Computational experiments}, we report the results of our computational experiments, and finally, in Section~\ref{Sec: Conclusion}, we present our conclusions.

\subsection{Notation} \label{Sec: notation}
We define $[n] \eqdef \{1,\hdots,n\}$, and denote by $\mathbbm{1}_\mc{S} \in \{0,1\}^n$ the indicator vector of $\mc{S} \subseteq [n]$.
We denote by $0_n\in\mathbb{R}^n, e_n\in\mathbb{R}^n, \mathbf{0}_n \in\mathbb{R}^{n\times n}, J_n \in\mathbb{R}^{n\times n}$ and $I_n \in\mathbb{R}^{n\times n}$ the all-zeros vector, all-ones vector, all-zeros matrix, all-ones matrix and identity matrix, respectively.
We sometimes omit the subscripts of these matrices when their size is clear from the context or irrelevant. 
We regularly work with matrices of size $(n+1)\times(n+1)$, in which case the first row and column are indexed by $0$.

The rank of a matrix $X$ is denoted by $\rank{X}$.
The trace inner product of two matrices $A, B \in$~$\mathbb{R}^{m\times n}$ is given by $\IP{A}{B} \eqdef \tr(A^\top B) = \sum_{i=1}^m \sum_{j=1}^n A_{ij} B_{ij}$.
A real symmetric matrix $X$ of size $n\times n$ is denoted by $X \in \mathbb{S}^n$, and we denote by $X \psd \textbf{0}$ that the symmetric matrix $X$ is positive semidefinite (PSD). 
The set of all PSD matrices of size $n\times n$ is denoted by $\mathbb{S}^n_+ \eqdef \{X\in \mathbb{S}^n: X \psd \textbf{0}\}$. 
Similarly, $X \nsd \textbf{0}$ indicates that matrix $X$ is negative semidefinite (NSD), and we define $\mathbb{S}^n_- \eqdef \{X\in \mathbb{S}^n: X \nsd \textbf{0}\}$. 
We write $X\geq \textbf{0}$ to denote that the matrix $X$ is elementwise nonnegative.

The operator $\text{diag}: \mathbb{R}^{n\times n} \to \mathbb{R}^n$ maps a matrix to the vector containing its diagonal entries.
Its adjoint operator $\text{Diag}: \mathbb{R}^{n} \to \mathbb{R}^{n\times n}$ maps a vector to the diagonal matrix with that vector as diagonal. 
For any $\mc{X} \subseteq \mathbb{R}^n$ and $u, w \in \mathbb{R}^n$, the $w$-weighted projection of $u$ on $\mc{X}$ is given by $\mc{P}_\mc{X}^w(u) \eqdef \argmin_{x \in \mc{X}} \left\{\sum_{i \in [n]} w_{i}(x_{i}-u_{i})^2\right\}$.
The (standard) projection of $u$ on $\mc{X}$ is given by $\mc{P}_\mc{X}(u) \eqdef \mc{P}_\mc{X}^{e_n}(u) = \argmin_{x \in \mc{X}} \norm{x-u}^2$.

\section{A BSDP formulation and SDP relaxation for the M$k$CS problem} \label{Sec: BSDP formulation}
Let $\mc{G} = (\mc{V},\mc{E})$ be a simple undirected graph with vertex set $\mc{V} = [n]$ and edge set $\mc{E}$, and let $k \in [n-1]$ be a given number of colors. 
A graph is said to be $k$-colorable if one can assign to each vertex one of the $k$ colors such that no two adjacent vertices have the same color. The M$k$CS problem is to find an induced $k$-colorable subgraph of $\mc{G}$ with maximum cardinality.
Recall that the subgraph of $\mc{G}$ induced by $\mc{S} \subseteq \mc{V}$ is the graph $\mc{G}[\mc{S}] = (\mc{S}, \mc{E}[\mc{S}])$ with vertex set $\mc{S}$ and edge set $\mc{E}[\mc{S}] \eqdef \{\{i,j\} \in \mc{E}: i,j \in \mc{S}\}$.

Let $P \in \{0,1\}^{n \times k}$ be a binary matrix variable such that $P_{ic}=1$ if and only if vertex $i \in [n]$ is assigned color $c \in [k]$. 
Then, the M$k$CS problem can be formulated as the following BLP problem: 
\begin{align} \tag{BLP$_k$} \label{BLP} &\begin{aligned}
    \alpha_k(\mc{G}) := \max \quad & \sum_{i \in [n]} \sum_{c \in [k]} P_{ic} \\
    \text{s.t.} \quad & P_{ic} + P_{jc} \leq 1 & \quad \forall \{i,j\} \in \mc{E}, \; c \in [k] \,\\
    & \sum_{c \in [k]}P_{ic} \leq 1 & \quad \forall i \in [n]\, \\
    & P_{ic} \in \{0,1\} & \quad \forall i \in [n], \; c \in [k].
\end{aligned} \end{align}

Next, we present a derivation from \cite{deMeijer2024BSDP} of a BSDP formulation of the M$k$CS problem using a matrix-lifting approach. 
The key idea of this approach is to substitute the matrix variable $P$ from \eqref{BLP} by another matrix variable $X$ such that $X = PP^\top$. 
This substitution can be performed as follows.
First, using that $P$ is binary, we have that $X_{ii} =\sum_{c \in [k]} P_{ic}^2 = \sum_{c \in [k]} P_{ic}$ for each $i \in [n]$, so we can rewrite the objective as $\IP{I_n}{X}$.
Second, by again exploiting the binarity of $P$, we can replace the constraints $P_{ic} + P_{jc} \leq 1$ for $\{i,j\} \in \mc{E}$ and $c \in [k]$, with the constraints $X_{ij} = \sum_{c \in [k]} P_{ic}P_{jc} = 0$ for $\{i,j\} \in \mc{E}$.
Finally, we consider the following set $\mc{D}^n_k$ that captures the remaining constraints on $P$ and the relation $X=PP^\top$:
\begin{equation}
    \mc{D}^n_k \eqdef \left\{X \in \mathbb{S}^n: X= PP^\top,\, \sum_{c \in [k]} P_{ic} \leq 1 \,\, \forall i \in [n], \,\, P \in \{0,1\}^{n \times k}\right\}. \label{BQP characterization 1}  
\end{equation}
Note that $\mc{D}^n_1$ coincides with the set of extreme points of the BQP introduced in \citep{Padberg1989BQP}.
As implied by Proposition 2 of \cite{Letchford2012BinarySDPMatrices} and Corollary~2 of \cite{deMeijer2024BSDP}, respectively, $\mc{D}^n_k$ can be described independently of the variable $P$ in either of the following two ways:
\begin{align}
    \mc{D}^n_k &= \left\{X \in \mathbb{S}^n: X \psd \textbf{0}, \,\,\rank{X} \leq k, \,\, X \in \{0,1\}^{n \times n}\right\} \label{BQP characterization 2}\\ 
    &= \left\{X \in \mathbb{S}^n: \aug{X} = \begin{pmatrix} k & \diag{X}^\top \\ \diag{X} & X \end{pmatrix} \psd \textbf{0}, \,\, X \in \{0,1\}^{n \times n}\right\}.  \label{BQP characterization 3}
\end{align}
Using the latter characterization, we obtain the following BSDP formulation of the M$k$CS problem:
\begin{align} \tag{BSDP$_k$} \label{BSDP} &\begin{aligned}
    \alpha_k(\mc{G}) = \max \quad & \IP{I_n}{X} \\
    \text{s.t.} \quad & X_{ij} = 0 \qquad \forall \{i,j\} \in \mc{E} \\
    & \aug{X} = \begin{pmatrix} k & \diag{X}^\top \\ \diag{X} & X \end{pmatrix} \psd \textbf{0} \\
    & X \in \{0,1\}^{n \times n}.
\end{aligned} \end{align}
We denote the feasible set of \eqref{BSDP} by $\mc{F}_{\eqref{BSDP}}$.
The binary matrix variable $X$ has the interpretation that a diagonal entry $X_{ii} = 1$ if and only if vertex $i \in [n]$ is assigned a color, and an off-diagonal entry $X_{ij} = 1$ if and only if vertices $i, j \in [n]$ with $i \neq j$ are assigned the same color. 
The appearance of the constant $k$ in the PSD constraint and the binarity of the matrix variable together ensure that $\rank{X} \leq k$, or equivalently, that at most $k$ colors are used.
Note that \eqref{BSDP} is free of color symmetry, whereas \eqref{BLP} is not.

For $k = 1$, i.e., the stable set problem, it is sufficient to impose integrality on the diagonal of $X$ to ensure also integrality on the off-diagonal entries. 
This follows from the nonnegativity of the determinants of the $3\times 3$ principle submatrices of $\aug{X}$, see e.g., Section 3.2 of \cite{Helmberg2000SDPForCO}.
This result does not hold in general for $k>1$.
For instance, consider the graph $\mc{G} = (\mc{V}, \mc{E})$ with $\mc{V} = \{1,2,3\}$ and $\mc{E} = \{\{2,3\}\}$ and let $k=2$. Then $X = \begin{psmallmatrix} 1 & 1/2 & 0\\ 1/2 & 1 & 0 \\ 0 & 0 & 0\end{psmallmatrix}$ does not satisfy integrality on all off-diagonal entries, while $X$ does satisfy all other constraints of \eqref{BSDP}.

By relaxing the integrality constraints on $X$ in \eqref{BSDP} and replacing them with the following bound constraints
\begin{equation}
    X \geq \textbf{0}_n, \quad \diag{X} \leq e_n, \label{eq bounds}
\end{equation}
we obtain our basic SDP relaxation:
\begin{align} \tag{SDP$_k$} \label{SDP} &\begin{aligned}
    \theta_k(\mc{G}) \eqdef \max \quad & \IP{I_n}{X} \\
    \text{s.t.} \quad & X_{ij} = 0 \qquad \forall \{i,j\} \in \mc{E} \\
    & \aug{X} = \begin{pmatrix} k & \diag{X}^\top \\ \diag{X} & X \end{pmatrix} \psd \textbf{0} \\
    & X \geq \textbf{0}_n, \,\, \diag{X} \leq e_n.
\end{aligned} \end{align}
We denote the feasible set of \eqref{SDP} by $\mc{F}_{\eqref{SDP}}$.
The SDP relaxation \eqref{SDP} was first proposed in~\cite{Kuryatnikova2022MkCSProblem}, where $\theta_k(\mc{G})$ is denoted by $\theta^3_k(\mc{G})$.
It was proved in that paper that \eqref{SDP} satisfies the Slater condition.
Observe that \eqref{SDP} is a Doubly Nonnegative relaxation, since it is an SDP relaxation over the set of nonnegative matrices. 
Moreover, note that the constraint $\diag{X} \leq e_n$ is sufficient to ensure that $X \leq J_n$. This follows from the nonnegativity of the determinants of the $2\times 2$ principle submatrices of $X$.
Furthermore, as shown in the following lemma, \eqref{SDP} dominates the LP relaxation of \eqref{BLP}.
\begin{lemma}
    For any graph $\mc{G}$ and $k \in [n-1]$, we have that $\theta_k(\mc{G}) \leq LP_k(\mc{G})$, where $LP_k(\mc{G})$ denotes the optimal value of the LP relaxation of \eqref{BLP}.
\end{lemma}
\begin{proof}
   Let $X \in \mc{F}_{\eqref{SDP}}$. It is not difficult to show that the matrix $P \in [0,1]^{n \times k}$ with $P_{ic} = X_{ii}/{k}$ for $i \in [n]$ and $c \in [k]$ is feasible to the LP relaxation of \eqref{BLP}, and that its objective value $\sum_{i \in [n]}\sum_{c \in [k]} P_{ic}$ equals $\IP{I_n}{X}$. 
   In particular, for any $\{i,j\} \in \mc{E}$ and $c \in [k]$, we have $P_{ic} + P_{jc} = \frac{1}{k}(X_{ii} + X_{jj}) \leq 1$, where the inequality follows from the constraint $\IP{aa^\top}{\aug{X}} \geq 0$ for $a = \begin{psmallmatrix} 1 \\ -\mathbbm{1}_{\{i,j\}} \end{psmallmatrix}$, which holds because $\aug{X} \psd \textbf{0}$.
\end{proof}

\section{Graph parameter $\theta_k(\mc{G})$} \label{Sec: Graph parameter}
We consider $\theta_k(\mc{G})$ as a graph parameter and derive several of its properties.
Our main result is that $\theta_k(\mc{G})$ is sandwiched between $\alpha_k(\mc{G})$ and the $k$-clique cover number.
Graphs for which these two values coincide for each induced subgraph are known as $k$-perfect graphs. 
Since $\theta_k(\mc{G})$ can be computed in polynomial time up to fixed precision, it follows that the M$k$CS problem is also polynomially solvable for $k$-perfect graphs.

\citet{Narasimhan1990GenTheta} introduced the generalized $\vartheta$-number of $\mc{G}$ as an eigenvalue bound for $\alpha_k(\mc{G})$, which can be computed as follows \cite{Alizadeh1995SDPGenTheta}:
\begin{align} \tag{$ \vartheta_k$} \label{generalizedTheta} &
\begin{aligned}
    \vartheta_k(\mc{G}) \eqdef
     \max \quad & \IP{J_n}{Z}  \\
    \text{s.t.} \quad & Z_{ij}=0  \qquad \forall \{i,j\} \in \mc{E}  \\
    &  \IP{I_n}{Z} =k \\
    & Z \succeq \mathbf{0}, \, I_n-Z\succeq \mathbf{0}.
\end{aligned} 
\end{align}    
\citet{Kuryatnikova2022MkCSProblem} strengthened $\vartheta_k(\mc{G})$ by adding nonnegativity constraints.
The resulting graph parameter is denoted by  $\vartheta_k'(\mc{G})$, and for $k=1$ it coincides with the Schrijver number {\cite{Schrijver1979}}, which is denoted in the literature by $\vartheta'(\mc{G})$. 
It is known that  $\vartheta'(\mc{G})=\theta_1(\mc{G})$ \cite{Kuryatnikova2022MkCSProblem}.
It was conjectured in the same paper that $\theta_k(\mc{G})$ is at least as strong as  $\vartheta'_k(\mc{G})$ for all $k\geq 1$.
The authors of \cite{Sinjorgo2022GenVartheta} proved that $\theta_k(\mc{G})$ and $\vartheta'_k(\mc{G})$ coincide for any strongly regular graph with parameters $(n,d,\lambda,\mu)$ and restricted eigenvalues $r\geq0$, $s<-1$, provided that $k < (n(r+1))/(r+n-d)$. 

Similarly to the (generalized) $\vartheta$- and $\vartheta'$-numbers,  $\theta_k(\mc{G})$ can also be regarded as a graph parameter. 
Let us first relate  $\theta_k(\mc{G})$ and  $\theta_1(\mc{G})$.
\begin{lemma} \label{Lemma1}
    For any undirected graph $\mc{G}$ with $n$ vertices, and an integer $k\in [n-1]$,  
    $$\theta_k(\mc{G}) \leq  k \theta_1(\mc{G}).$$
\end{lemma}
\begin{proof}
    The result follows trivially for $k=1$.
    Let $k\geq 2$, and let $X^k$ be an optimal solution to~\eqref{SDP}.
    Let $X^1:= \frac{1}{k}X^k$, then $X^1$ is feasible for (SDP$_1$). 
    Therefore,
    $$\theta_k(\mc{G}) = \IP{I_n}{X^k} = k\IP{I_n}{\frac{1}{k} X^k} =  k\IP{I_n}{X^1} \leq k  \theta_1(\mc{G}),$$
    where the inequality follows from the fact that $X^1$ need not be an optimal solution to (SDP$_1$).
\end{proof}

Next, we relate $\theta_k(\mathcal{G})$ to the chromatic number of graph $\mathcal{G}$, denoted by $\chi(\mathcal{G})$.
\begin{lemma}\label{Lemma2}
    Let $\mc{G}=(\mc{V},\mc{E})$  be a given graph and  $k \in \mathbb{N}$ such that $k\geq \chi( \mc{G} )$, then $\theta_k(\mc{G}) = |\mc{V}|$.
\end{lemma}
\begin{proof}
    Since  $k\geq \chi( \mc{G} )$, we have $\alpha_k(\mc{G}) = |\mc{V}|$, and therefore 
    $|\mc{V}| \leq  \theta_k(\mc{G})$. 
    The fact that $\theta_k(\mc{G})\leq |\mc{V}|$ trivially follows from the constraint $\diag{X} \leq e_n$.
\end{proof}

We summarize the results of the previous lemmas below.
\begin{corollary} \label{Cor1}
For a given graph $\mc{G}=(\mc{V},\mc{E})$ with $n$ vertices and an integer $k\in [n-1]$, we have
$$\theta_k(\mc{G}) \leq \min \left \{ k \vartheta'(\mc{G}), |\mc{V}|\right \},$$
where $\vartheta'(\mc{G})$ is the Schrijver number.
\end{corollary}

It was shown in~\cite{Sinjorgo2022GenVartheta} that the sequence   
$\left ( \vartheta_k (\mc{G}) \right )_k$ is increasing in $k$.
We show below that  $\left (\theta_k(\mc{G}) \right )_k$ is non-decreasing in $k$.

\begin{lemma}
    For any graph $\mc{G}$ and $k \in \mathbb{N}$, $\theta_k(\mc{G}) \leq \theta_{k+1}(\mc{G}).$
\end{lemma}
\begin{proof}
    Let $X^{k}$ be an optimal solution to \eqref{SDP}.
    It is not difficult to verify that $X^{k}$ is also a feasible solution to the SDP relaxation $(SDP_{k+1})$.
    Therefore, $\theta_k(\mc{G}) =  \IP{I_n}{X^k} \leq \theta_{k+1}(\mc{G})$.
\end{proof}

The following summarizes the relationships among several graph parameters:
\begin{align}  \label{summarize inequalities}
    \alpha_k(\mc{G}) \leq \theta_k(\mc{G}) \leq  k \vartheta'(\mc{G}) \leq k \vartheta(\mc{G}) \leq  \chi_k( \overline{\mc{G}}),
\end{align}
where $\overline{\mc{G}}$ denotes the complement graph of $\mc{G}$, and 
$\chi_k(\overline{\mc{G}})$ the minimum number of colors needed for a valid $k$-multicoloring of $ \overline{\mc{G}}$.
Recall that a valid $k$–multicoloring of a graph is an assignment of $k$ distinct colors to each vertex in the graph such that two adjacent vertices are assigned disjoint sets of colors.
Clearly, $\chi_1(\mc{G}) = \chi(\mc{G}).$
While the first three inequalities in \eqref{summarize inequalities} follow from the previous discussion, for the proof of the last inequality, see e.g., \cite{Sinjorgo2022GenVartheta}.

Note that the parameter $\chi_k( \overline{\mc{G}})$ can be larger than the number of vertices in $\mc{G}$, while $\alpha_k(\mc{G})$ and  $\theta_k(\mc{G})$ are bounded above by $n$.
Therefore, one may ask whether there is a more appropriate graph parameter than $\chi_k( \overline{\mc{G}})$ that serves as an upper bound on $\theta_k(\mc{G})$, thereby forming a sandwich relation. 
We consider here the following parameter
\begin{align}\label{Psik}
   \Psi_k({\mc{G}}) := \min \left \{  k  \chi(\mc{G} -\mc{S}) + |\mc{S}| \,:\, \mc{S} \subseteq \mc{V} \right \},
\end{align}
that was studied by \citet{Greene1976}, and \citet{GreeneKleitman1976}.
Here, $\mc{G} - \mc{S}$ denotes the subgraph of $\mc{G}$ induced by $\mc{V} \setminus \mc{S}$.

Clearly, $ \Psi_k({\mc{G}}) \leq n$ for any $k\geq 1$, and $\Psi_1({\mc{G}})=  \chi({\mc{G}})$.
\citet{Narasimhan1989ThesisMkCSProblem} also proposed replacing $\chi_k(\overline{\mc{G}})$ with $\Psi_k(\overline{\mc{G}})$
in the sandwich inequality $\alpha_k(\mc{G}) \leq \vartheta_k(\mc{G}) \leq  \chi_k(\overline{\mc{G}})$, but they were unable to prove that $\vartheta_k({\mc{G}}) \leq  \Psi_k(\overline{\mc{G}})$.
\citet{Oliveira2024} introduced a spectral nonconvex graph parameter, which is sandwiched between $\alpha_k(\mc{G})$ and $\Psi_k( \overline{\mc{G}})$ and is difficult to compute.
In the sequel, we show that $\theta_k(\mc{G})$ is bounded above by  $\Psi_k(\overline{\mc{G}})$, resulting in the desired sandwich relation.
Let us first note that
$$ \Psi_k(\overline{\mc{G}}) =  \min \left \{  k  \chi( \overline{\mc{G}} -\mc{S}) + |\mc{S}| \,:\, \mc{S} \subseteq \mc{V} \right \} = 
\min \left \{  k  \overline{\chi}( {\mc{G} -\mc{S}}) + |\mc{S}| \,:\, \mc{S} \subseteq \mc{V} \right \}, $$
where $\overline{\chi}(\mc{G})$ denotes the clique cover number of $\mc{G}$. 
Here we exploit the fact that $\overline{\chi}(\mc{H}) = \chi(\overline{\mc{H}})$ for any graph~$\mc{H}$.
Therefore, $\Psi_k(\overline{\mc{G}})$ may be seen as a generalization of the clique cover number.
We therefore refer to the minimization problem above as the $k$-clique cover problem.

First, we prove the following result.
\begin{proposition}\label{prop:compareTheta3}
    Let $\mc{G}=(\mc{V},\mc{E})$  be a graph with $n$ vertices, and let $\mc{S} \subseteq \mc{V}$. Then,  
    $$ \theta_k(\mc{G}) \leq \theta_k(\mc{G} -\mc{S}) + |\mc{S}|,$$
    for any $k \in [n-1]$.
\end{proposition}
\begin{proof}
    If $\mc{S}=\mc{V}$, the inequality follows trivially.
    Let us therefore assume that $\mc{S} \neq \mc{V}$. 
    Let $X$ be an optimal solution to \eqref{SDP}, and let $X_{\mc{S}}$ and $X_{\mc{V}- \mc{S}}$ be the submatrices of $X$ indexed by the vertices in the set $\mc{S}$ and $\mc{V}\setminus \mc{S}$, respectively.
    It is not difficult to verify that $X_{\mc{V}- \mc{S}}$ is feasible for the SDP relaxation \eqref{SDP} induced by $\mc{V}\setminus \mc{S}$.
    Now, we consider the following two cases.
    
    Assume that  $k < n-|\mc{S}|$. Then, 
    $$ \theta_k(\mc{G}) = \IP{I_n}{X}  = \IP{I_{n-|\mc{S}|}}{X_{\mc{V}- \mc{S}}} + 
    \IP{I_{|\mc{S}|}}{X_{\mc{S}}} \leq \theta_k(\mc{G}-\mc{S}) + |\mc{S}|. $$
    Assume now that $k\geq n-|\mc{S}|$. 
    Then, we proceed similarly as above:
    $$ \theta_k(\mc{G}) = \IP{I_n}{X} \leq \theta_k(\mc{G}-\mc{S}) + |\mc{S}| = n-|\mc{S}| + |\mc{S}| = n, $$
    where we exploited that $\theta_k(\mc{G}-\mc{S}) = n-|\mc{S}|$ since $k\geq n-|\mc{S}|$. 
\end{proof}

We are now ready to prove our main result.
\begin{theorem}
    Let $\mc{G}=(\mc{V},\mc{E})$ be a graph with $n$ vertices, and let $k \in [n-1]$ be an integer. Then, 
    $$\theta_k(\mc{G}) \leq  \Psi_k(\overline{\mc{G}}).$$
\end{theorem}
\begin{proof}
We show that $\theta_k(\mc{G}) \leq k  \overline{\chi}({\mc{G} -\mc{S}}) + |\mc{S}|$ for all $\mc{S} \subseteq \mc{V}$.
To do that we consider three cases.
Assume first that $\mc{S}=\mc{V}$. Then, $k \overline{\chi}({\mc{G}} -\mc{S}) + |\mc{S}| =  k\cdot 0 +n$, from where it follows $\theta_k(\mc{G}) \leq n$.

Let us now assume that $\mc{S}\subseteq \mc{V}$ and $k < n-|\mc{S}|$.
Then,
$$\theta_k(\mc{G}) \leq \theta_k(\mc{G} -\mc{S}) + |\mc{S}| 
\leq k    \theta_1(\mc{G} - \mc{S})  + |\mc{S}|
\leq k  \overline{\chi}({\mc{G} - \mc{S}}) + |\mc{S}|,$$
where we used Proposition~\ref{prop:compareTheta3} for the first inequality, and Lemma~\ref{Lemma1} for the second inequality. 
The last inequality follows from the fact that $\theta_1(\mc{G} -\mc{S}) = \vartheta'(\mc{G} -\mc{S}) \leq \vartheta(\mc{G} -\mc{S})$ and the well-known Lov\'asz sandwich theorem 
$\alpha(\mc{H})\leq \vartheta(\mc{H})\leq \chi (\overline{\mc{H}})$ for any graph $\mc{H}$, see also \eqref{summarize inequalities} for $k=1$.

Assume now that $\mc{S}\varsubsetneq \mc{V}$ and $k \geq n-|\mc{S}|$. Then,
$$\theta_k(\mc{G}) \leq n = n-|\mc{S}| + |\mc{S}| \leq k  \overline{\chi}({\mc{G} - \mc{S}}) + |\mc{S}|,$$
where we exploit Proposition~\ref{prop:compareTheta3} and the fact that $\overline{\chi}({\mc{G} - \mc{S}})\geq 1$.
\end{proof}

Finally, we can establish the following sandwich relation 
$$ \alpha_k(\mc{G}) \leq \theta_k(\mc{G}) \leq  \Psi_k(\overline{\mc{G}}),$$
for $k\in [n-1]$.
Since $\theta_k(\mc{G}) $ can be computed in polynomial time up to fixed precision, it follows  that the M$k$CS problem is also solvable in polynomial time for any graph for which $\alpha_k(\mc{G})= \Psi_k(\overline{\mc{G}})$.
In particular, this is true for $k$-perfect graphs. 
Recall that a graph $\mc{G}$ is called $k$-perfect if $\alpha_k ( \mc{G}[\mc{S}]) =  \Psi_k(\overline{\mc{G}[\mc{S}]})$ holds for all $\mc{S}\subseteq \mc{V}$ \cite{lovasz1983perfect}.
A $k$-perfect graph for $k=1$ is known as a perfect graph. 
Perfect graphs were introduced by Claude Berge in the early sixties, and the Strong Perfect Graph Theorem by \citet{Chudnovsky2006SPGT} fully characterizes perfect graphs.
On the other hand, only a few classes of $k$-perfect graphs for $k>1$ have been identified. 
\citet{GreeneKleitman1976} proved that comparability graphs are $k$-perfect for every integer $k\geq 1$, and so are their complements, as shown by \citet{Greene1976}. 
A graph that is $k$-perfect for every integer $k\geq 1$ is called a totally perfect graph.
\citet{lovasz1983perfect} showed that line graphs of bipartite graphs are totally perfect, and \citet{Berge1992} proved that all balanced graphs are totally perfect.

Our result resolves the problem of finding an efficient algorithm for computing $\alpha_k(\mc{G}) $ and $\Psi_k(\overline{\mc{G}})$ for $k$-perfect graphs where $k>1$.
\citet{Grtschel1984PolynomialAF} showed that there is a polynomial time algorithm for computing these parameters for perfect graphs, thus for the case $k=1$.
To the best of our knowledge, no such result was previously known for $k>1$.
\begin{theorem}
    Let $\mc{G}$  be a $k$-perfect graph and $k\geq 1$.  
    Then, $\alpha_k(\mc{G})$ and $\Psi_k(\overline{\mc{G}})$ can be computed in polynomial time up to fixed precision.
\end{theorem} 

\section{Valid inequalities} \label{Sec: Valid inequalities}
In this section, we derive several families of valid inequalities that strengthen \eqref{SDP}.
We first propose generalized BQP inequalities (Section~\ref{Sec: Boolean quadric polytope inequalities}), and then derive generalized rank inequalities (Section~\ref{Sec: Rank inequalities}). 
We further consider computationally tractable special cases of our new valid inequalities that concern triangles (Section~\ref{Sec: Triangle inequalities}), cliques (Section \ref{Sec: Clique inequalities}) and odd holes (Section~\ref{Sec: Odd-hole inequalities}).

\subsection{Generalized Boolean quadric polytope inequalities} \label{Sec: Boolean quadric polytope inequalities}
We first present valid inequalities that generalize the \textit{cut inequalities} and \textit{clique inequalities} of \citet{Padberg1989BQP} for the BQP. 
Namely, we prove that the following inequalities are valid for the set $\mc{D}^n_k$, where $k \in [n-1]$:
\begin{alignat}{3}
    & \sum_{i \in \mc{S}}\sum_{j \in \mc{S}'} X_{ij} \leq \sum_{i \in \mc{S}'}X_{ii} + \sum_{i,j \in \mc{S}: i<j} X_{ij} + \sum_{i,j \in \mc{S}': i<j} X_{ij} &\quad& \forall \mc{S}, \mc{S}' \subseteq \mc{V} \text{ with } \mc{S} \cap \mc{S}' = \emptyset \label{eq BQP1} \\
    & \sum_{i \in \mc{S}} X_{ii} \leq \sum_{i,j \in \mc{S}: i<j} X_{ij} + k &\quad& \forall \mc{S} \subseteq \mc{V}. \label{eq BQP2}
\end{alignat}
Inequalities~\eqref{eq BQP1} coincide with Padberg's cut inequalities, and inequalities~\eqref{eq BQP2} coincide with a subset of Padberg's clique inequalities when $k=1$.
Therefore, we will further refer to inequalities~\eqref{eq BQP1} and~\eqref{eq BQP2} as the \textit{generalized BQP inequalities}, whose validity we prove next.
\begin{theorem}\label{thm:bqp1}
    Inequalities~\eqref{eq BQP1} and~\eqref{eq BQP2} are valid for $\mc{D}^n_k$, 
    where $k \in [n-1]$.
\end{theorem}
\begin{proof}
    Let $X \in \mc{D}^n_k$, which implies that there exists a matrix $P \in \{0,1\}^{n \times k}$ such that $X = PP^\top$. 
    
    First, we prove the validity of inequality~\eqref{eq BQP1} for $\mc{S}, \mc{S}' \subseteq \mc{V}$ with $\mc{S} \cap \mc{S}' = \emptyset$.
    Let $c \in [k]$. 
    It can easily be shown that $xy \leq y + \binom{x}{2} +
    \binom{y}{2}$ for all $x, y \in \mathbb{N}_0 {\eqdef \mathbb{N} \cup \{0\}}$. 
    Taking $x = \sum_{i \in \mc{S}} P_{ic} \in \mathbb{N}_0$ and $y = \sum_{i \in \mc{S}'} P_{ic} \in \mathbb{N}_0$, we obtain 
    $$\sum_{i \in \mc{S}} P_{ic} \sum_{j \in \mc{S}'} P_{jc} \leq \sum_{i \in \mc{S}'} P_{ic} + \binom{\sum_{i \in \mc{S}} P_{ic}}{2} + \binom{\sum_{i \in \mc{S}'} P_{ic}}{2} = \sum_{i \in \mc{S}'} P_{ic} + \sum_{i,j \in \mc{S}: i < j} P_{ic}P_{jc} + \sum_{i,j \in \mc{S}': i < j} P_{ic}P_{jc},$$
    where we also use binarity of $P$ to get $\binom{\sum_{i \in \mc{S}} P_{ic}}{2} = \sum_{i,j \in \mc{S}, i < j} P_{ic} P_{jc}$ and $\binom{\sum_{i \in \mc{S}'} P_{ic}}{2} = \sum_{i,j \in \mc{S}', i < j} P_{ic} P_{jc}$.
    Summing all $k$ such inequalities for $c \in [k]$, we obtain 
    $$\sum_{i \in \mc{S}} \sum_{j \in \mc{S}'} \sum_{c \in [k]} P_{ic}P_{jc} \leq  \sum_{i \in \mc{S}'}\sum_{c \in [k]} P_{ic} + \sum_{i,j \in \mc{S}: i < j} \sum_{c \in [k]} P_{ic}P_{jc} + \sum_{i,j \in \mc{S}': i < j} \sum_{c \in [k]} P_{ic}P_{jc}.$$ 
    Finally, we can substitute $X_{ii} = \sum_{c \in [k]} P_{ic}$ for each $i \in [n]$ and $X_{ij} = \sum_{c \in [k]} P_{ic}P_{jc}$ for each $i, j \in [n]$ ($i \neq j$) to obtain inequality~\eqref{eq BQP1} for $\mc{S}$, $\mc{S}'$.  

    Next, we prove the validity of inequality~\eqref{eq BQP2} for $\mc{S} \subseteq \mc{V}$ in a similar fashion.
    For each $c \in [k]$, we have that
    $$\sum_{i \in \mc{S}} P_{ic} \leq \binom{\sum_{i \in \mc{S}} P_{ic}}{2} + 1 = \sum_{i, j \in \mc{S}: i < j} P_{ic} P_{jc} + 1,$$
    where we used that $x \leq \binom{x}{2} + 1$ for all $x \in \mathbb{N}_0$. 
    Summing all $k$ such inequalities for $c \in [k]$, and replacing the $P$-variables by the $X$-variables, we obtain inequality~\eqref{eq BQP2} for $\mc{S}$.
\end{proof}
Note that in the proof of Theorem~\ref{thm:bqp1} we did not require that $\sum_{c \in [k]}P_{ic} \leq 1$ for all $i \in [n]$.
Furthermore, it is not difficult to verify that inequalities~\eqref{eq BQP2} for sets $\mc{S} \subseteq \mc{V}$ with $|\mc{S}| \leq k$ are already implied by bound constraints~\eqref{eq bounds}.

\subsubsection{Generalized triangle inequalities} \label{Sec: Triangle inequalities}
The exponential number of generalized BQP inequalities makes it intractable to include all of them as cutting planes for strengthening \eqref{SDP}.
Therefore, in practice, one might restrict to the following generalized BQP inequalities that involve only three vertices:
\begin{alignat}{3}
    &X_{i\ell} + X_{j\ell} \leq X_{\ell\ell} + X_{ij}  &\quad& \forall \{i,j,\ell\} \subseteq \mc{V} \label{eq T1} \\
    &X_{ii} + X_{jj} + X_{\ell\ell} \leq X_{ij} + X_{i\ell} + X_{j\ell} + k &\quad& \forall \{i,j,\ell\} \subseteq \mc{V}. \label{eq T2}
\end{alignat}
Inequalities~\eqref{eq T1} follow from inequalities~\eqref{eq BQP1} with $\mc{S} = \{i,j\}$ and $\mc{S}' = \{\ell\}$, and inequalities~\eqref{eq T2} follow from inequalities~\eqref{eq BQP2} with $\mc{S} = \{i,j,\ell\}$. 
Inequalities~\eqref{eq T1}, and inequalities~\eqref{eq T2} for $k=1$, are known as \textit{triangle inequalities}. 
They were introduced in \cite{Padberg1989BQP} and have been used to strengthen various relaxations, including SDP relaxations for the stable set problem \citep{Gruber2003SSRelaxations} and the strongest known SDP relaxation for the M$k$CS problem \citep{Kuryatnikova2022MkCSProblem}.
We also considered the generalized BQP inequalities that involve only two vertices. 
However, preliminary experiments showed that it is not beneficial to add those to \eqref{SDP}.

\subsection{Generalized rank inequalities} \label{Sec: Rank inequalities}
In this section, we introduce four sets of inequalities that generalize the rank inequalities proposed by \citet{Nemhauser1974PropertiesPolyhedra}, see also \cite{Letchford2020CliqueAndNodalInequalities}. 
We extend these inequalities from BLP to BSDP, and from $k=1$ (i.e., the stable set problem) to arbitrary $k \in [n-1]$.
Specific types of rank inequalities were already extended to semidefinite programs in \cite{Pucher2023SSAndColoringBounds, Pucher2025PracticalExperience}.  
These works focus on the M$k$CS for $k=1$ and subgraphs inducing cliques, odd holes, and odd antiholes, while we consider arbitrary values of $k$ and general induced subgraphs.

To state our \textit{generalized rank inequalities}, we require the following definition.
\begin{definition} 
    Let $\mc{S} \subseteq \mc{V}$ be a subset of vertices  and let $\kappa \in [k]$ be a number of colors.
    The $\kappa$-rank of $\mc{S}$ is defined as $r_\kappa(\mc{S}) \eqdef \alpha_\kappa(\mc{G}[\mc{S}])$, which is the maximum number of vertices that can be colored in $\mc{G}[\mc{S}]$ with at most $\kappa$ colors. 
    We also define $r_0(\mc{S}) \eqdef 0$.
\end{definition}
The $1$-rank of a set coincides with the notion of rank used in the rank inequalities for the stable set problem \cite{Letchford2020CliqueAndNodalInequalities, Nemhauser1974PropertiesPolyhedra}.
Throughout this section, we further use that for any feasible solution $X \in \mc{F}_{\eqref{BSDP}}$ and subset of vertices $\mc{S} \subseteq \mc{V}$, the sum $\sum_{i \in \mc{S}} X_{ii}$ gives the number of vertices in $\mc{S}$ that are colored, the sum $\sum_{i,j \in \mc{S}: i < j} X_{ij}$ gives the number of pairs of vertices in $\mc{S}$ that have the same color, and the sum $\sum_{i \in \mc{S}}X_{i\ell}$ gives the number of vertices in $\mc{S}$ that have the same color as some external vertex $\ell \in \mc{V} \setminus \mc{S}$.

Our first class of generalized rank inequalities for $\mc{F}_\eqref{BSDP}$ follows directly from the definition of~$r_k(S)$:
\begin{alignat}{3}
    & \sum_{i \in \mc{S}} X_{ii} \leq r_k(\mc{S}) &\quad& \forall \mc{S} \subseteq \mc{V}. \label{eq rankInternal1}
\end{alignat}

The next inequalities provide an upper bound on the number of  pairs of vertices in a set $\mc{S} \subseteq \mc{V}$ that may be colored with the same color, when using the available $k$ colors:
\begin{alignat}{3}
    & \sum_{i,j \in \mc{S}: i < j} X_{ij} \leq \sum_{\kappa=1}^k\binom{r_\kappa(\mc{S})-r_{\kappa-1}(\mc{S})}{2} &\quad& \forall \mc{S} \subseteq \mc{V}. \label{eq rankInternal2}
\end{alignat}
To prove the validity of these inequalities, we use the following proposition.
\begin{proposition}[Proposition B.2.~of \cite{Marshall2011Majorization}] \label{Weak majorization lemma}
    Let $f$ be a continuous non-decreasing convex function, and let $x = (x_i)_{i=1}^n$ and $y = (y_i)_{i=1}^n$ be two sequences such that $y$ weakly majorizes $x$. 
    This means that $\sum_{i = 1}^m x_{(i)} \leq \sum_{i = 1}^m y_{(i)}$ for each $m \in [n]$, where $x_{(1)} \geq \hdots \geq x_{(n)}$ and $y_{(1)} \geq \hdots \geq y_{(n)}$ denote the elements of $x$ and $y$ in non-increasing order.
    Then $\sum_{i = 1}^n f(x_i) \leq \sum_{i = 1}^n f(y_i)$.
\end{proposition}

We are now ready to prove the following theorem.
\begin{theorem}
    Inequalities~\eqref{eq rankInternal2} are valid for $\mc{F}_\eqref{BSDP}$.
\end{theorem}
\begin{proof}
    Let $f(x) = \max\left\{\frac{1}{2}x(x-1),0\right\}$, which is continuous, non-decreasing and convex in $x$.  
    Consider any assignment of at most $k$ colors to the vertices in some set $\mc{S} \subseteq \mc{V}$ (possibly leaving some vertices uncolored) and let $b = (b_c)_{c=1}^k$ be the sequence where $b_c \geq 0$ denotes the number of vertices in $\mc{S}$ that are assigned color $c \in [k]$.
    Furthermore, let $\delta = (\delta_\kappa)_{\kappa=1}^k$ be the sequence where $\delta_\kappa \eqdef r_\kappa(\mc{S})-r_{\kappa-1}(\mc{S}) \geq 0$ for each number of colors $\kappa \in [k]$ used.
    For each $m \in [k]$, we have $\sum_{c=1}^m b_{(c)} \leq r_m(\mc{S}) = \sum_{\kappa = 1}^m \delta_\kappa \leq \sum_{\kappa = 1}^m \delta_{(\kappa)}$, so $\delta$ weakly majorizes $b$.
    Hence, Proposition~\ref{Weak majorization lemma} applies. 
    Finally, let $X \in \mc{F}_{\eqref{BSDP}}$ be a feasible solution corresponding to the considered color assignment. We then have that $\sum_{i, j \in \mc{S}: i<j} X_{ij} = \sum_{c = 1}^k \binom{b_c}{2} = \sum_{c = 1}^k f(b_c) \leq \sum_{\kappa=1}^k f(\delta_\kappa) = \sum_{\kappa = 1}^k \binom{\delta_\kappa}{2}$.
\end{proof}

Next, we present inequalities that provide an upper bound on the number of vertices in a set $\mc{S} \subseteq \mc{V}$ that may have the same color as an external vertex $\ell \in \mc{V} \setminus \mc{S}$.
If $\ell$ is colored, then by the definition of $r_1(\mc{S})$, $\ell$ can have the same color as up to $r_1(\mc{S})$ vertices from $\mc{S}$.
This leads to the following valid inequalities for $\mc{F}_{\eqref{BSDP}}$:
\begin{alignat}{3}
    & \sum_{i \in \mc{S}} X_{i\ell} \leq r_1(\mc{S}) X_{\ell\ell} &\quad& \forall \mc{S} \subseteq \mc{V}, \ell \in \mc{V} \setminus \mc{S}. \label{eq rankExternal1}
\end{alignat}

Finally, the following complementary inequalities provide a lower bound on the number of vertices in a set $\mc{S} \subseteq \mc{V}$ that may have the same color as an external vertex $\ell \in \mc{V} \setminus \mc{S}$, given that $k$ colors are available:
\begin{alignat}{3}
    & \sum_{i \in \mc{S}} X_{ii} + \left(r_k(\mc{S}) - r_{k-1}(\mc{S})\right) X_{\ell\ell} \leq \sum_{i \in \mc{S}} X_{i\ell} + r_k(\mc{S}) &\quad& \forall \mc{S} \subseteq \mc{V}, \ell \in \mc{V} \setminus \mc{S}. \label{eq rankExternal2}
\end{alignat}
We prove the validity of these inequalities in the following theorem.
\begin{theorem}
    Inequalities~\eqref{eq rankExternal2} are valid for $\mc{F}_\eqref{BSDP}$.
\end{theorem}
\begin{proof}
    Let $\mc{S} \subseteq \mc{V}$, let $\ell \in \mc{V} \setminus \mc{S}$, and let $X \in \mc{F}_\eqref{BSDP}$.
    Observe that when coloring more than $r_{k-1}(\mc{S})$ vertices in $\mc{S}$ with $k$ colors, then each of the $k$ available colors appears at least $\sum_{i \in \mc{S}}X_{ii}-r_{k-1}(\mc{S}) > 0$ times in $\mc{S}$. 
    Indeed, if some color were to be assigned to fewer than $\sum_{i \in \mc{S}}X_{ii}-r_{k-1}(\mc{S})$ vertices, then the total number of colored vertices in $\mc{S}$ would be less than $\sum_{i \in \mc{S}}X_{ii}$, since the remaining $k-1$ colors can together cover at most $r_{k-1}(\mc{S})$ vertices.
    Thus, depending on the number of colored vertices in $\mc{S}$, if vertex $\ell$ is colored, then $\ell$ has the same color as at least $\max\left\{\sum_{i \in \mc{S}}X_{ii}-r_{k-1}(\mc{S}), 0\right\}$ of the vertices in $\mc{S}$.
    Therefore, we have
    \begin{align*}
    \sum_{i \in \mc{S}} X_{i\ell} &\geq \max\left\{\sum_{i \in \mc{S}}X_{ii}-r_{k-1}(\mc{S}), 0\right\} X_{\ell\ell} \geq \left(\sum_{i \in \mc{S}}X_{ii}-r_{k-1}(\mc{S})\right)X_{\ell\ell} \\
    &\geq \left(\sum_{i \in \mc{S}}X_{ii}-r_{k-1}(\mc{S})\right) - \left(r_{k}(\mc{S})-r_{k-1}(\mc{S})\right)\left(1-X_{\ell\ell}\right),    
    \end{align*}
    where we used for the last inequality that $\sum_{i \in \mc{S}}X_{ii} \leq r_k(\mc{S})$ and that $X_{\ell\ell} \in \{0,1\}$. 
    Rewriting the obtained inequality gives the desired result.  
\end{proof}

Note that it is intractable to add all generalized rank inequalities~\eqref{eq rankInternal1}-\eqref{eq rankExternal2} as cuts to strengthen \eqref{SDP}, since there are exponentially many such inequalities, and, moreover, computing the $\kappa$-rank of an arbitrary set $\mc{S}$ is $\mc{NP}$-hard.
Therefore, in the following sections we restrict attention to sets $\mc{S}$ that are cliques and odd holes, as these sets admit explicit expressions for the $\kappa$-ranks.
In our computational experiments, we further limit the size of the cliques and odd holes considered.

\subsubsection{Generalized clique inequalities} \label{Sec: Clique inequalities}
This section concerns the generalized rank inequalities~\eqref{eq rankInternal1}-\eqref{eq rankExternal2} for subsets of $\mc{V}$ that are cliques. 
Recall that a clique in $\mc{G}$ is a set of vertices $\mc{Q} \subseteq \mc{V}$ that induces a complete subgraph, i.e., every pair of vertices in $\mc{Q}$ is adjacent in $\mc{G}$. 
For any clique $\mc{Q} \subseteq \mc{V}$, we have that $r_\kappa(\mc{Q}) = \min\{\kappa, |\mc{Q}|\}$ for each $\kappa \in [k]$.

We first show that inequalities \eqref{eq rankInternal1} and~\eqref{eq rankInternal2} for cliques are already implied by the constraints of \eqref{SDP}.
\begin{lemma} \label{Redundant clique inequalities}
    Let $X \in \mc{F}_\eqref{SDP}$ and let $\mc{Q} \subseteq \mc{V}$ be a clique.
    Then $\sum_{i \in \mc{Q}} X_{ii} \leq \min\{k,|\mc{Q}|\}$ and $\sum_{i,j \in \mc{Q}: i < j} X_{ij} =0$.
\end{lemma}
\begin{proof}
    We have that $X_{ij} = 0$ for each $i,j \in \mc{Q}$ ($i \neq j$), and thus it follows that $\sum_{i,j \in \mc{Q}: i<j} X_{ij} = 0$. Moreover, since $\aug{X} \psd \textbf{0}$, we have that $\IP{aa^\top}{\aug{X}} \geq 0$ for all $a \in \mathbb{R}^{n+1}$. 
    In particular, for $a = \begin{pmatrix} 1 \\ -\mathbbm{1}_{\mc{Q}}\end{pmatrix}$ we get $k -\sum_{i \in \mc{Q}} X_{ii} + 2\sum_{i, j \in \mc{Q}: i < j} X_{ij} \geq 0$. 
    Using that $\sum_{i,j \in \mc{Q}: i<j} X_{ij} = 0$, we obtain the inequality $\sum_{i \in \mc{Q}} X_{ii} \leq k$. Finally, the inequality $\sum_{i \in \mc{Q}} X_{ii} \leq |\mc{Q}|$ follows from the constraint $\diag{X} \leq e_n$.
\end{proof}

In contrast, inequalities~\eqref{eq rankExternal1} and~\eqref{eq rankExternal2} for cliques are not implied by the constraints $\mc{F}_\eqref{SDP}$.
For reference, these inequalities read as follows:
\begin{alignat}{3}
    & \sum_{i \in \mc{Q}} X_{i\ell} \leq X_{\ell\ell} &\quad& \forall \text{ clique } \mc{Q} \subseteq \mc{V}, \ell \in \mc{V} \setminus \mc{Q} \label{eq cliqueExternal1} \\
    & \sum_{i \in \mc{Q}} X_{ii} + X_{\ell\ell} \leq \sum_{i \in \mc{Q}} X_{i\ell} + k &\quad& \forall \text{ clique } \mc{Q} \subseteq \mc{V} \text{ with } |\mc{Q}| \geq k,  \ell \in \mc{V} \setminus \mc{Q}. \label{eq cliqueExternal2}
\end{alignat}
We will refer to these inequalities as the \textit{generalized clique inequalities}.
Note that \eqref{eq cliqueExternal1} and \eqref{eq cliqueExternal2} can also be derived from the generalized BQP inequalities~\eqref{eq BQP1} and~\eqref{eq BQP2}, respectively.
To see this, take $\mc{S} = \mc{Q}$ and $\mc{S}' = \{\ell\}$ for inequality~\eqref{eq BQP1} and $\mc{S} = \mc{Q} \cup \{\ell\}$ for inequality~\eqref{eq BQP2}, and use Lemma~\ref{Redundant clique inequalities}. 
Note that inequalities~\eqref{eq cliqueExternal2} for cliques $\mc{Q} \subseteq \mc{V}$ with $|\mc{Q}| < k$ are already implied by bound constraints~\eqref{eq bounds}.

\citet{Pucher2023SSAndColoringBounds} also considered inequalities~\eqref{eq cliqueExternal1} and~\eqref{eq cliqueExternal2} for the stable set problem, i.e., for the M$k$CS problem for $k=1$. 
They found that inequalities~\eqref{eq cliqueExternal2} can be strengthened by extending the external vertex $\ell$ to a second clique $\mc{Q}'$. 
Applying this idea to the M$k$CS problem for $k \in [n-1]$, we obtain the following generalized clique inequalities:
\begin{alignat}{3}
    & \sum_{i \in \mc{Q}} X_{ii} + \sum_{j \in \mc{Q}'} X_{jj} \leq \sum_{i \in \mc{Q}}\sum_{j \in \mc{Q}'} X_{ij} + k &\quad& \forall \text{ cliques } \mc{Q}, \mc{Q}' \subseteq \mc{V} \text{ with } \mc{Q} \cap \mc{Q}' = \emptyset, |\mc{Q}| + |\mc{Q}'| > k. \label{eq cliqueUnion}
\end{alignat}
Note that inequalities~\eqref{eq cliqueUnion} are also a special case of generalized BQP inequalities~\eqref{eq BQP2} (take $\mc{S} = \mc{Q} \cup \mc{Q}'$ and use Lemma~\ref{Redundant clique inequalities}). 
Moreover, observe that inequalities~\eqref{eq cliqueExternal2} are a special case of inequalities~\eqref{eq cliqueUnion}. 
Preliminary experiments confirmed that inequalities~\eqref{eq cliqueUnion} result in stronger bounds than inequalities~\eqref{eq cliqueExternal2} alone.

As implied by the following proposition, when using generalized clique inequalities~\eqref{eq cliqueExternal1} as cuts to strengthen \eqref{SDP}, it suffices to consider these inequalities for maximal cliques only.

\begin{proposition} \label{Prop: only maximal cliqueExternal1}
    Let $X \in \mc{F}_{\eqref{SDP}}$  satisfy inequality~\eqref{eq cliqueExternal1} for a clique $\mc{Q} \subseteq \mc{V}$ and  an external vertex $\ell \in \mc{V} \setminus \mc{Q}$. 
    Then, for each $h \in \mc{Q}$, $X$ also satisfies inequality~\eqref{eq cliqueExternal1} for the clique $\mc{Q} \setminus \{h\}$ and the vertex~$\ell$.
\end{proposition}

\begin{proof}
    Using that $\sum_{i\in \mc{Q}} X_{i\ell} \leq X_{\ell\ell}$ and $X_{h\ell} \geq 0$, we find that $\sum_{i \in \mc{Q} \setminus \{h\}} X_{i\ell} = \sum_{i \in \mc{Q}} X_{i\ell} - X_{h\ell} \leq X_{\ell\ell} - X_{h\ell} \leq X_{\ell\ell}$, which is exactly inequality~\eqref{eq cliqueExternal1} for $\mc{Q} \setminus \{h\}$ and $\ell$.
\end{proof}

Furthermore, once inequalities~\eqref{eq cliqueExternal1} have been added to \eqref{SDP}, it is sufficient to consider inequalities~\eqref{eq cliqueUnion} for pairs of maximal cliques, as implied by the next proposition.
\begin{proposition} \label{Prop: only maximal cliqueUnion}
    Let $X \in \mc{F}_\eqref{SDP}$  satisfy inequality~\eqref{eq cliqueUnion} for disjoint cliques $\mc{Q}, \mc{Q}' \subseteq \mc{V}$. 
    Suppose that $X$ satisfies inequality~\eqref{eq cliqueExternal1} for the clique $\mc{Q}$ and a vertex $h \in \mc{Q}'$. 
    Then, $X$ also satisfies inequality~\eqref{eq cliqueUnion} for the  cliques $\mc{Q}$ and $\mc{Q}'\setminus \{h\}$.
\end{proposition}
\begin{proof}
    Using that $\sum_{i \in \mc{Q}} X_{ii} + \sum_{j \in \mc{Q}'} X_{jj} \leq \sum_{i \in \mc{Q}}\sum_{j \in \mc{Q}'} X_{ij} + k$ and $\sum_{i \in \mc{Q}} X_{ih} \leq X_{hh}$, we find~$\sum_{i \in \mc{Q}} X_{ii}$
    $+ \sum_{j \in \mc{Q}' \setminus \{h\}} X_{jj}
    = \sum_{i \in \mc{Q}} X_{ii} + \sum_{j \in \mc{Q}'} X_{jj} - X_{hh}
    \leq \sum_{i \in \mc{Q}}\sum_{j \in \mc{Q}'} X_{ij} + k - X_{hh}
    = \sum_{i \in \mc{Q}} \sum_{j \in \mc{Q}' \setminus \{h\}} X_{ij} + \sum_{i \in \mc{Q}} X_{ih} + k - X_{hh}
    = \sum_{i \in \mc{Q}} \sum_{j \in \mc{Q}' \setminus \{h\}} X_{ij} + k + \left(\sum_{i \in \mc{Q}} X_{ih} - X_{hh}\right) \leq \sum_{i \in \mc{Q}} \sum_{j \in \mc{Q}' \setminus \{h\}} X_{ij} + k$, 
    which is exactly inequality~\eqref{eq cliqueUnion} for $\mc{Q}$ and $\mc{Q}'\setminus \{h\}$.
\end{proof}
Note that this proposition also applies when interchanging the roles of $\mc{Q}$ and $\mc{Q}'$.

\subsubsection{Generalized odd-hole inequalities} \label{Sec: Odd-hole inequalities}
We now consider the generalized rank inequalities~\eqref{eq rankInternal1}-\eqref{eq rankExternal2} for subsets of $\mc{V}$ that are odd holes.
We first recall some definitions.
A cycle in $\mc{G}$ is a set of vertices $\mc{C} = \{v_1, v_2, \hdots, v_{|\mc{C}|}\} \subseteq \mc{V}$ with $|\mc{C}| \geq 3$ such that $\{v_i,v_{i+1}\} \in \mc{E}$ for all $i \in [|\mc{C}|-1]$ and $\{v_{|\mc{C}|}, v_1\} \in \mc{E}$.
Additional edges in a cycle that connect nonconsecutive vertices are called chords. A chordless cycle is called a hole. A cycle or hole is called odd (respectively even) if its cardinality is odd (respectively even).
Odd holes of cardinality $3$ are also cliques, so we only consider here odd holes of cardinality at least $5$.
Observe that for any odd hole $\mc{C} \subseteq \mc{V}$, the maximum number of vertices that can be colored in $\mc{G}[\mc{C}]$ with $\kappa \in [k]$ colors equals
$$r_\kappa(\mc{C}) = \min\left\{\frac{|\mc{C}|-1}{2} \kappa, |\mc{C}|\right\} = \begin{cases}
    \frac{|\mc{C}|-1}{2} & \text{if $\kappa=1$} \\
    |\mc{C}|-1 & \text{if $\kappa=2$} \\
    |\mc{C}| & \text{otherwise}.
\end{cases}$$

In preliminary experiments, we considered the strength of \eqref{SDP} when adding inequalities~\eqref{eq rankInternal1}-\eqref{eq rankExternal2} for odd holes. 
Among these inequalities, the ones of type~\eqref{eq rankExternal1} led to the largest improvements in the SDP bounds. 
Motivated by this, we focus on these inequalities here. In the following, we will refer to them as \textit{generalized odd-hole inequalities}, and they read as follows:
\begin{alignat}{3}
    & \sum_{i \in \mc{C}} X_{i\ell} \leq \frac{|\mc{C}|-1}{2} X_{\ell\ell} &\quad& \forall \text{ odd hole } \mc{C} \subseteq \mc{V}, \ell \in \mc{V} \setminus \mc{C}. \label{eq HoleExternal1}
\end{alignat}
These inequalities were first proposed by \citet{Pucher2025PracticalExperience} for SDP relaxations for the stable set problem.

Let us briefly consider even holes.
For any even hole $\mc{C} \subseteq \mc{V}$ we have that $r_1(\mc{C}) = \frac{|\mc{C}|}{2}$ and $r_\kappa(\mc{C}) = |\mc{C}|$ for all $\kappa \geq 2$. 
Therefore, for any even hole $\mc{C}$ we can replace the factor $\frac{|\mc{C}|-1}{2}$ in inequalities~\eqref{eq HoleExternal1} by $\frac{|\mc{C}|}{2}$ to obtain valid inequalities.
However, as proven in the next proposition, such inequalities are already implied by triangle inequalities~\eqref{eq T1}.
\begin{proposition} \label{Prop: no even holes}
    Let $X \in \mc{F}_{\eqref{SDP}}$ satisfy all triangle inequalities~\eqref{eq T1}. 
    Let $\mc{C} \subseteq \mc{V}$, $|\mc{C}| \geq 4$ be an even hole and let $\ell \in \mc{V} \setminus \mc{C}$ be an external vertex. 
    Then $X$ satisfies $\sum_{i \in \mc{C}} X_{i\ell} \leq \frac{|\mc{C}|}{2} X_{\ell\ell}$.
\end{proposition}
\begin{proof}
    Without loss of generality, we can label the vertices of $\mc{C}$ as $1, 2, \hdots, |\mc{C}|$ so that $\{i, i+1\} \in \mc{E}$ for each odd $i \in \{1, \hdots, |\mc{C}|-1\}$. 
    Therefore, for each such $i$, we can combine the triangle inequality~\eqref{eq T1} for triplet $\{i,i+1, \ell\}$ with the constraint $X_{i,i+1} = 0$ to obtain that $X_{i\ell} + X_{i+1, \ell} \leq X_{\ell\ell}$. 
    Summing all $\frac{|\mc{C}|}{2}$ such inequalities, we obtain $\sum_{i \in \mc{C}} X_{i\ell} \leq \frac{|\mc{C}|}{2} X_{\ell\ell}$.
\end{proof}

In a similar vein, note that inequalities~\eqref{eq HoleExternal1} also hold for odd cycles containing chords, as such cycles $\mc{C}$ satisfy $r_1(\mc{C}) \leq \frac{|\mc{C}|-1}{2}$. 
However, we do not need to consider such inequalities to strengthen~\eqref{SDP}, as implied by the following proposition.

\begin{proposition} \label{Prop: no cycles with chords}
    Let $X \in \mc{F}_{\eqref{BSDP}}$, let $\mc{C} \subseteq \mc{V}$ be an odd cycle with $|\mc{C}| \geq 5$ that contains at least one chord, and let $\ell \in \mc{V}\setminus \mc{C}$ be an external vertex. 
    Then, the valid inequality $\sum_{i \in \mc{C}} X_{i\ell} \leq \frac{|\mc{C}|-1}{2} X_{\ell\ell}$ can be expressed as the sum of an odd-hole inequality of the form~\eqref{eq HoleExternal1} and a number of triangle inequalities~\eqref{eq T1}.   
\end{proposition}
\begin{proof}
    Without loss of generality, we can label the vertices of $\mc{C}$ as $1, 2, \dots, |\mc{C}|$ so that $\{i,i+1\} \in \mc{E}$ for each $i \in \mc{C}$ (where we use the convention that addition is taken modulo $|\mc{C}|$ when dealing with cycles) and that $\{1,p\} \in \mc{E}$ for some even $p \in \{4, \hdots, |\mc{C}|-1\}$.
    Note that $\mc{C}' = \{p, \hdots, |\mc{C}|, 1\}$ is an odd cycle, so $\sum_{i \in \{p, \hdots, |\mc{C}|, 1\}} X_{i\ell} \leq \frac{|\mc{C}|-p+1}{2}X_{\ell\ell}$. 
    Moreover, for each even $i \in \{2, \hdots, p-2\}$ we can combine the triangle inequality~\eqref{eq T1} for triplet $\{i,i+1, \ell\}$ with the constraint $X_{i,i+1} = 0$ to obtain that $X_{i\ell} + X_{i+1, \ell} \leq X_{\ell\ell}$. 
    Summing all $\frac{p-2}{2}$ such inequalities and the inequality for the odd cycle $\mc{C}'$, we obtain $\sum_{i \in \mc{C}} X_{i\ell} \leq \frac{|\mc{C}|-1}{2}X_{\ell\ell}$.
    If $\mc{C}'$ is chordless, the proposition follows. 
    Otherwise, the same decomposition can be applied recursively.
\end{proof}

\section{The cutting-plane ADMM} \label{Sec: CP-ADMM}
This section describes the CP-ADMM that is used to compute valid upper bounds on $\alpha_k(\mc{G})$.
In Section~\ref{Sec: ADMM scheme}, we present the ADMM scheme for solving SDP relaxations of \eqref{BSDP}, and in Section~\ref{Sec: Cutting-plane strategy}, we explain how this scheme is embedded into a cutting-plane framework.
Whereas our implementation incorporates the valid inequalities derived in the previous section, our algorithm is described independently of the specific inequalities used, emphasizing that the CP-ADMM can handle any valid inequality of the form $\IP{A}{X} \leq b$.

\subsection{The ADMM scheme} \label{Sec: ADMM scheme}
At a given iteration of the CP-ADMM, we consider a fixed index set $\mc{I}$ of valid inequalities for \eqref{BSDP}, and apply the ADMM to solve
\begin{equation} \tag{SDP$_k$-$\mc{I}$} \label{SDP-I}
    \theta_k^\mc{I}(\mc{G}) \eqdef  \max\left\{ \IP{I_n}{X}: X \in \mc{F}_\eqref{SDP}, \IP{A_i}{X} \leq b_i \,\, \forall i \in \mc{I}\right\}.
\end{equation}

To apply the ADMM, we first reformulate \eqref{SDP-I} by separating the affine constraints from the PSD constraint.
To this end, let
\begin{equation}
    \aug{\mc{X}}_\mc{I} \eqdef \left\{\begin{pmatrix} k & \diag{X}^\top \\ \diag{X} & X\end{pmatrix} \in \mathbb{S}^{n+1}: X_{ij} = 0 \,\,\forall \{i,j\} \in \mc{E},\,\, \textbf{0}_n \leq X \leq J_n,\,\, \IP{A_i}{X} \leq b_i \,\,\forall i \in \mc{I}\right\}, \label{eq: def affine set}
\end{equation}
and let $\aug{I} \eqdef \begin{pmatrix} 0 & 0_n^\top \\ 0_n & I_n \end{pmatrix}$,
from where it follows that
\begin{equation}
    \theta_k^\mc{I}(\mc{G}) = \max \left\{\IP{\aug{I}}{\aug{X}}: \aug{X}=Y,  \,\, \aug{X} \in \aug{\mc{X}}_\mc{I},  \,\, Y \in \mathbb{S}_+^{n+1}\right\}. \label{ADMM split}
\end{equation}
Note that $\aug{\mc{X}}_{\mc{I}}$ includes the redundant upper bound constraints $X_{ij} \leq 1$ for off-diagonal entries $i, j \in [n]$ with $i \neq j$. 
Despite being redundant for \eqref{SDP-I}, these constraints do fasten the convergence of the ADMM, see e.g., \cite{deMeijer2025ADMMQMSTP, deMeijer2023ADMMPartitioning, Li2021PRSMMinCut, Oliveira2018ADMMForQAP}.

Next, we consider the augmented Lagrangian function of~\eqref{ADMM split} with respect to the coupling constraint $\aug{X}=Y$, which is defined as
\begin{equation}
    L_{\beta}(\aug{X}, Y, \Lambda) \eqdef \IP{\aug{I}}{\aug{X}} - \IP{\Lambda}{\aug{X}-Y} - \frac{\beta}{2}\norm{\aug{X}-Y}^2, \label{ADMM augmented Lagrangian}
\end{equation}
where $\Lambda \in \mathbb{S}^{n+1}$ is the dual variable associated to the coupling constraint $\aug{X}=Y$, and $\beta > 0$ is a fixed penalty parameter (we specify our parameter values in Table~\ref{tab:parameters CP-ADMM} in Section~\ref{Sec: Computational experiments}). 
Starting from an initial triple $(\aug{X}^0, Y^0, \Lambda^0)$, and given some stepsize parameter $\gamma$, the ADMM iteratively updates the current triple $(\aug{X}^t, Y^t, \Lambda^t)$ using the following update scheme:
\begin{align}
    \aug{X}^{t+1} \eqdef& \argmax_{\aug{X} \in \aug{\mc{X}}_\mc{I}} L_\beta(\aug{X}, Y^t, \Lambda^t) \label{ADMM-X-update}\\
    Y^{t+1} \eqdef& \argmax_{Y \in \mathbb{S}_+^{n+1}} L_\beta(\aug{X}^{t+1}, Y, \Lambda^{t})\label{ADMM-Y-update}\\
    \Lambda^{t+1} \eqdef& \Lambda^t + \gamma \cdot \beta (\aug{X}^{t+1}-Y^{t+1}). \label{ADMM-L-update}
\end{align}
When $\gamma \in (0, \frac{1+\sqrt{5}}{2})$, the iterates $(\aug{X}^t, Y^t)$ converge with rate $\mc{O}(1/t)$ to an optimal solution of~\eqref{ADMM split}, see \cite{He2016ConvergenceSymmetricADMM}.

The efficiency of the ADMM depends on the complexity of the subproblems~\eqref{ADMM-X-update} and~\eqref{ADMM-Y-update}.
It can be shown that these problems reduce to the following projection problems, see \cite{Oliveira2018ADMMForQAP}:
\begin{align}
    &\argmax_{\aug{X} \in \aug{\mc{X}}_\mc{I}} L_\beta(\aug{X}, Y^t, \Lambda^t) 
    = \argmin_{\aug{X} \in \aug{\mc{X}}_\mc{I}} \norm{\aug{X}-\left(Y^t+\frac{1}{\beta}\left(\aug{I} - \Lambda^t\right)\right)}^2
    = \mc{P}_{\aug{\mc{X}}_\mc{I}}\left(Y^t + \frac{1}{\beta}\left(\aug{I} - \Lambda^t\right)\right) \label{ADMM-X-projection}\\
    &\argmax_{Y \in \mathbb{S}^{n+1}_+} L_\beta(\aug{X}^{t+1}, Y, \Lambda^t) = \argmin_{Y \in \mathbb{S}^{n+1}_+} \norm{Y- \left(\aug{X}^{t+1} + \frac{1}{\beta}\Lambda^t\right)}^2 = \mc{P}_{\mathbb{S}^{n+1}_+}\left(\aug{X}^{t+1} + \frac{1}{\beta}\Lambda^t\right).\label{ADMM-Y-projection}
\end{align}
The projection of a symmetric matrix onto $\mathbb{S}^{n+1}_+$ is obtained through a spectral decomposition. 
This decomposition is carried out in single precision instead of double precision to reduce its computational cost, see \cite{Sinjorgo2025ADMMStability}.

We next discuss the projection onto $\aug{\mc{X}}_\mc{I}$.
Note that every matrix $\aug{X} \in \aug{\mc{X}}_\mc{I}$ is uniquely determined by the entries $\aug{X}_{ij}$ with $1 \leq i \leq j\leq n, \{i,j\} \not \in \mc{E}$. 
Therefore, we can reduce the projection problem onto $\aug{\mc{X}}_\mc{I}$ to a \textit{weighted} projection problem (see Section~\ref{Sec: notation}) where we only consider the $m \eqdef n + \binom{n}{2}-|\mc{E}|$ free entries. 
Concretely, given a matrix $\aug{U} = \begin{pmatrix} u_{00} & u_0^\top \\ u_0 & U \end{pmatrix} \in \mathbb{S}^{n+1}$, its projection onto $\aug{\mc{X}}_\mc{I}$ proceeds as follows:
\begin{enumerate}
    \item To account for the zeroth row and column, replace the diagonal of $U$ by $\frac{1}{3}\diag{U} + \frac{2}{3}u_0$, and then construct the vector $u \in \mathbb{R}^m$ that contains only the free entries of the newly obtained matrix.
    \item Find $u' \eqdef \mc{P}^w_{\mc{X}^{vec}_\mc{I}}(u)$, where weight vector $w \in \{2,3\}^m$ assigns weight $2$ to upper-triangular entries (accounting for the lower-triangular entries) and weight $3$ to diagonal entries (accounting for the entries on the zeroth row and column), and
    \begin{equation}
        \mc{X}^{vec}_\mc{I} \eqdef \left \{x \in \mathbb{R}^m: 0_m \leq x \leq e_m,\quad a_i^\top x \leq b_i \quad \forall i \in \mc{I} \right \}.
    \end{equation}
    The vectors $a_i \in \mathbb{R}^m$ are obtained from the matrices $A_i \in\mathbb{S}^{n+1}$.
    \item Map the vector $u' \in \mc{X}^{vec}_\mc{I}$ to the corresponding matrix in $\aug{\mc{X}}_\mc{I}$.
\end{enumerate}
To find $\mc{P}^w_{\mc{X}^{vec}_\mc{I}}(u)$, we first rewrite $\mc{X}^{vec}_\mc{I}$ as the intersection of simpler sets:
\begin{equation}
    \mc{X}^{vec}_\mc{I} = [0,1]^m \cap \left(\bigcap_{i \in \mc{I}}\mc{H}_i \right), \quad \text{where } \mc{H}_i = \left\{x \in \mathbb{R}^m: a_i^\top x \leq b_i\right\} \quad \forall i \in \mc{I}.
\end{equation}
These sets are all easy to project onto. Indeed, the weighted projection onto the box is given by
\begin{equation}
    \mc{P}^w_{[0,1]^m}(x) = \mc{P}_{[0,1]^m}(x) = \left(\min\{\max\{x_j, 0\}, 1\}\right)_{j \in [m]}, \label{eq: box projection}
\end{equation}
and for each valid inequality $i \in \mc{I}$, the weighted projection onto its corresponding halfspace is given by
\begin{equation}
    \mc{P}^w_{\mc{H}_i}(x) = x - \frac{\max\{a_i^\top x - b_i, 0\}}{a_i^\top\Diag{w}^{-1} a_i} \Diag{w}^{-1}a_i. \label{eq: cut projection}
\end{equation}
When $\mc{I} = \emptyset$, we have that $\mc{X}^{vec}_\mc{I} = [0,1]^m$, so the projection is straightforward.
Otherwise, we apply Dykstra's cyclic projection algorithm \citep{Boyle1986Dykstra}, which is an iterative algorithm for projecting onto the intersection of convex sets.
This algorithm was successfully used within the CP-ADMM algorithms from \cite{deMeijer2025ADMMQMSTP, deMeijer2021QCCP, deMeijer2023ADMMPartitioning}.
Importantly, these papers accelerate Dykstra's algorithm by partitioning the valid inequalities into clusters of non-overlapping inequalities so that each iteration of Dykstra's algorithm requires only one projection per cluster, rather than one projection per inequality.
We follow this approach, and refer to the aforementioned papers for the details.

\subsection{The cutting-plane framework} \label{Sec: Cutting-plane strategy}
We are now ready to present the full algorithm in which the ADMM is embedded in a cutting-plane framework, see Algorithm~\ref{Alg: ADMM}.

\begin{algorithm}[H]
\caption{CP-ADMM}
\label{Alg: ADMM}
\begin{algorithmic}[1]
\footnotesize
\Require{$\mc{G} = (\mc{V}, \mc{E})$, $k$, parameter values from Table~\ref{tab:parameters CP-ADMM} in Section~\ref{Sec: Computational experiments}}
\State Initialization: set $\aug{X}^0 \leftarrow \begin{pmatrix} k & e_n^\top \\ e_n & I_n \end{pmatrix}$, $Y^0 \leftarrow \begin{pmatrix} k & e_n^\top \\ e_n & I_n \end{pmatrix}$, $\Lambda^0 \leftarrow \mathbf{0}_{n+1}$, $t \leftarrow 0$, $\mc{I} \leftarrow \emptyset$
\While{the outer stopping criteria are not met} \Comment{see Section~\ref{Sec: Stopping criteria}}
    \While{the inner stopping criteria are not met} \Comment{see Section~\ref{Sec: Stopping criteria}}
        \State $\aug{X}^{t+1} \leftarrow \mc{P}_{\aug{\mc{X}}_\mc{I}}\left(Y^t + \frac{1}{\beta}\left(\aug{I} - \Lambda^t\right)\right)$
        \State $Y^{t+1} \leftarrow \mc{P}_{\mathbb{S}^{n+1}_+}\left(\aug{X}^{t+1} + \frac{1}{\beta}\Lambda^t\right)$
        \State $\Lambda^{t+1} \leftarrow \Lambda^t + \gamma \cdot \beta (\aug{X}^{t+1}-Y^{t+1})$
        \State $t \leftarrow t+1$
    \EndWhile
    \State \textbf{end}
    \State $UB \leftarrow \max_{\aug{X} \in \aug{\mc{X}}_\mc{I}}\IP{\aug{I}-\mc{P}_{\mathbb{S}^{n+1}_{-}}(\Lambda^t)}{\aug{X}}$  \Comment{see Section~\ref{Sec: Valid upper bounds}}
    \State Identify inequalities that are violated by $\aug{X}^{t}$ and add a selection of those to $\mc{I}$ \Comment{see Section~\ref{Sec: Cut separation and selection}}
\EndWhile
\State \textbf{end}
\Ensure{$UB$, $\aug{X}^t$}
\end{algorithmic}
\end{algorithm}

At each iteration of the outer while-loop (further referred to as an \textit{outer iteration}), we apply the ADMM to approximately solve \eqref{SDP-I} for a given set of valid inequalities $\mc{I}$.
The ADMM is executed within the inner while-loop until its stopping criteria are met (see Section~\ref{Sec: Stopping criteria}), yielding a final ADMM iterate $(\aug{X}^t, Y^t, \Lambda^t)$ for the current set $\mc{I}$. 
Subsequently, we check whether $\aug{X}^t$ violates any valid inequalities that are not yet included in $\mc{I}$.
If so, a selection of those is added to $\mc{I}$ (see Section~\ref{Sec: Cut separation and selection}), after which we apply the ADMM to solve the strengthened relaxation in the next outer iteration.
This procedure is repeated until some stopping criterion for the outer while-loop is met (see Section~\ref{Sec: Stopping criteria}).
Along the way, we find stronger valid upper bounds on $\alpha_k(\mc{G})$ using a method that is described in Section~\ref{Sec: Valid upper bounds}.

The algorithm is initialized with $\mc{I} = \emptyset$, meaning that the first outer iteration amounts to solving \eqref{SDP}. 
This initial iteration is significantly simpler than the subsequent ones, as the projection onto $\mc{X}^{vec}_\emptyset$ is a projection onto the box and does not require Dykstra's cyclic projection algorithm.
Therefore, initially the most time-consuming step is the projection onto $\mathbb{S}^{n+1}_+$, but once we start adding cuts, the bottleneck becomes Dykstra's algorithm. 
In order to improve computational efficiency, we employ a warm-start strategy where each outer iteration is initialized with the final solution from the previous outer iteration.

\subsubsection{Stopping criteria} \label{Sec: Stopping criteria}
We terminate an inner while-loop when the primal and dual residual satisfy
$$\max\left\{\frac{\norm{X^{t+1}-Y^{t+1}}}{1 + \norm{X^{t+1}}}, \beta \frac{\norm{X^{t+1}-X^t}}{1 + \norm{X^{t+1}}}\right\} \leq \eps_\texttt{ADMM},$$
see e.g., \cite{Oliveira2018ADMMForQAP}, 
or when we reach a maximum of \texttt{maxInnerIter} iterations for the inner while-loop.
In the final outer iteration, we decrease $\eps_\texttt{ADMM}$ (see also \cite{deMeijer2023ADMMPartitioning}) and increase \texttt{maxInnerIter} to obtain a more accurate final valid upper bound.

We terminate the outer while-loop (and thus the entire algorithm) when
(i) the current valid upper bound rounded down to the nearest integer equals some known lower bound,
(ii) the improvement of the current valid upper bound over the previous one is less than \texttt{minImpr},
(iii) there are fewer than \texttt{minIneq} violated inequalities,
or (iv) a time limit \texttt{timeLimitGlobal} is reached.

\subsubsection{Valid upper bounds} \label{Sec: Valid upper bounds}
We obtain valid upper bounds on $\alpha_k(\mc{G})$ from the ADMM's output using a procedure proposed by \citet{Oliveira2018ADMMForQAP}.
The dual function of~\eqref{ADMM split}, with dual variable $\Lambda \in \mathbb{S}^{n+1}_-$, is given by
\begin{equation}
    d(\Lambda)
    \eqdef \max_{\aug{X} \in \aug{\mc{X}}_\mc{I}, Y \in \mathbb{S}^{n+1}_+} \IP{\aug{I}}{\aug{X}} - \IP{\Lambda}{\aug{X}-Y}
    = \max_{\aug{X} \in \aug{\mc{X}}_\mc{I}} \IP{\aug{I}-\Lambda}{\aug{X}}. \label{ADMM valid UB} 
\end{equation}
It follows from weak duality that $d(\Lambda) \geq \theta_k^{\mc{I}}(\mc{G})$
for any $\Lambda \in \mathbb{S}^{n+1}_-$.
In particular, this holds for $\Lambda = \mc{P}_{\mathbb{S}^{n+1}_{-}}(\Lambda^t)$, where $\Lambda^t$ is any intermediate iterate of the dual variable in the ADMM.

Note that $d(\mc{P}_{\mathbb{S}^{n+1}_{-}}(\Lambda^t))$ can be computed in polynomial time as $\max_{\aug{X} \in \aug{\mc{X}}_\mc{I}} \IP{\aug{I}-\mc{P}_{\mathbb{S}^{n+1}_{-}}(\Lambda^t)}{\aug{X}}$ is a linear program and $\mc{P}_{\mathbb{S}^{n+1}_{-}}(\Lambda^t)$ can be found through a spectral decomposition of $\Lambda^t$. When $\mc{I} = \emptyset$, this linear program has a closed-form solution, namely the matrix $\aug{X} = \begin{psmallmatrix} k & \diag{X}^\top \\ \diag{X} & X\end{psmallmatrix}$ with $X$ given by
\begin{equation} 
        X_{ij} = 
        \begin{cases}
            0 &\text{if } \{i,j\} \in \mc{E}, \\
            \mathbbm{1}_{\{C_{ii} + 2C_{i0} > 0\}} &\text{if } i=j, \\
            \mathbbm{1}_{\{C_{ij} > 0\}} &\text{otherwise,}
        \end{cases} \quad\forall i, j \in [n],
\end{equation}
where $C = \aug{I}-\mc{P}_{\mathbb{S}^{n+1}_{-}}(\Lambda^t)$.

\subsubsection{Cut separation and selection} \label{Sec: Cut separation and selection}
We consider the following classes of valid inequalities and corresponding separation procedures:
\begin{itemize}
    \item Generalized triangle inequalities~\eqref{eq T1} and~\eqref{eq T2}: we enumerate all $\binom{n}{3}$ triplets of vertices.
    \item Generalized clique inequalities~\eqref{eq cliqueExternal1} and~\eqref{eq cliqueUnion}: to keep separation tractable, we initially restrict to maximal cliques up to size $5$ and all cliques of size $6$, where our focus on maximal cliques is justified by Propositions~\ref{Prop: only maximal cliqueExternal1} and~\ref{Prop: only maximal cliqueUnion}. 
    We enumerate these cliques using a variant of the Bron-Kerbosch algorithm \cite{Bron1973MaxCliques} with a time limit \texttt{timeLimitCliques}.
    For inequalities~\eqref{eq cliqueExternal1} we consider at most \texttt{maxCliques} cliques per {outer iteration} (where we randomly select a subset of the cliques if there are more than \texttt{maxCliques} options).
    Similarly, for inequalities~\eqref{eq cliqueUnion} we consider at most \texttt{maxCliquePairs} pairs of cliques per outer iteration.

    Whenever we identify a violated inequality of the form~\eqref{eq cliqueExternal1} for a non-maximal clique $\mc{Q}$ of size $6$ and external vertex $\ell$, we greedily extend the clique to a larger clique to obtain a stronger cut.     
    Namely, starting from $\mc{Q}$, we iteratively add the vertex $i$ with the highest value of $X_{i \ell}$ among the vertices that are adjacent to all vertices in the current clique, continuing until the clique is maximal.
    \item Generalized odd-hole inequalities~\eqref{eq HoleExternal1}: to keep separation tractable, we restrict to $5$-holes. 
    Recall from Propositions~\ref{Prop: no even holes} and~\ref{Prop: no cycles with chords} that we do not need to consider even holes or cycles with chords. 
    We use a time limit \texttt{timeLimitHoles} for enumerating the $5$-holes, and consider at most \texttt{maxHoles} 5-holes per outer iteration.
\end{itemize}
At every outer iteration of the CP-ADMM, we add up to \texttt{maxIneq} of the most violated inequalities, where an inequality only counts as violated if its violation is at least \texttt{minViol}.
Preliminary experiments showed that the generalized clique inequalities~\eqref{eq cliqueExternal1} are the most effective cuts among those considered. 
Therefore, we implemented a two-phase approach in which we first consider only inequalities~\eqref{eq cliqueExternal1}. 
Once either the improvement of the current valid upper bound over the previous one is less than \texttt{minImprPhase1} or there are fewer than \texttt{minIneqPhase1} violated inequalities of type~\eqref{eq cliqueExternal1}, we transition to a second phase where the remaining inequality types are also considered.

Finally, we limit the maximum number of new cuts in which each variable may appear to\linebreak \texttt{maxCutsPerVar}. This promotes cut diversity and makes it easier to partition the cuts into a small number of clusters of non-overlapping cuts, speeding up the algorithm considerably. To the best of our knowledge, we are the first to implement this feature in the context of the CP-ADMM.

\section{The integer ADMM} \label{Sec: INT-ADMM}
In this section, we present the INT-ADMM for computing feasible solutions to \eqref{BSDP}. Our algorithm is based on the $\ell_p$-Box ADMM from \cite{Wu2019LpBoxADMM}. 
That algorithm is designed to find feasible solutions to arbitrary binary optimization programs, and was successfully used e.g., in \cite{Jiao2021PenalizedLpBox} and \cite{Wu2018LpBoxUsage}.
The $\ell_p$-Box ADMM relies on the observation that, for any $p > 0$,
\begin{equation}
    \{0,1\}^n = [0,1]^n \cap \left\{x \in \mathbb{R}^n: \norm{x - \frac{1}{2}e}_p^p = \frac{n}{2^p}\right\}. \label{integer = box union sphere}
\end{equation}
That is, the set of binary vectors can be written as the intersection of the unit box with the $\ell_p$-sphere centered at $\frac{1}{2}e$ with radius $\frac{n^{1/p}}{2}$. 
For $p=2$, there exists a closed-form expression for projecting onto a sphere.

Extending this idea to \eqref{BSDP}, we first split its constraints into three components, each admitting an efficient projection: we have a variable $\aug{X}$ that satisfies the affine (including box) constraints, a variable $Y$ that satisfies the PSD constraint, and in addition, a third variable $Z$ that satisfies a sphere constraint.
In more detail, given that we work in $\mathbb{S}^{n+1}$ rather than $\mathbb{R}^n$, and also that entry $(0,0)$ should be equal to $k$, we consider the $\ell_2$-sphere in $\mathbb{S}^{n+1}$ centered at $C \eqdef \frac{1}{2}J_{n+1} + \begin{pmatrix} k & 0_n^\top \\ 0_n & \textbf{0}_n \end{pmatrix}$ with radius $(n+1)/2$:
\begin{equation}
    \mc{Z} = \left\{Z \in \mathbb{S}^{n+1} : \norm{Z - C}^2 = \frac{(n+1)^2}{4}\right\}.
\end{equation}
Note that $\mc{Z}$ is a non-convex set, which impacts the convergence of the INT-ADMM, as explained in Section~\ref{Sec: Convergence INT-ADMM}.
Then, letting $\aug{\mc{X}} \eqdef \aug{\mc{X}}_\emptyset$ (see~\eqref{eq: def affine set}), we can rewrite \eqref{BSDP} as follows:
\begin{equation}
    \alpha_k(\mc{G}) = \max \left\{\IP{\aug{I}}{\aug{X}}: \aug{X}=Y, \aug{X} = Z, \aug{X} \in \aug{\mc{X}}, Y \in \mathbb{S}_+^{n+1}, Z \in \mc{Z}\right\}. \label{INT-ADMM split}
\end{equation}

Next, we consider the augmented Lagrangian function of~\eqref{INT-ADMM split} with respect to the two coupling constraints $\aug{X} = Y$ and $\aug{X} = Z$, which is given by
\begin{equation}
    L_{\beta_Y, \beta_Z}(\aug{X}, Y, Z, \Lambda, M) \eqdef \IP{\aug{I}}{\aug{X}} - \IP{\Lambda}{\aug{X}-Y} - \IP{M}{\aug{X}-Z} - \frac{\beta_Y}{2}\norm{\aug{X}-Y}^2 - \frac{\beta_Z}{2}\norm{\aug{X}-Z}^2, \label{INT-ADMM augmented Lagrangian}
\end{equation}
where $\Lambda, M \in \mathbb{S}^{n+1}$ are the dual variables associated to the coupling constraints $\aug{X}=Y$ and $\aug{X}=Z$, respectively, and $\beta_Y, \beta_Z > 0$ are fixed penalty parameters.
We then obtain the following update rules:
\begin{align}
    \aug{X}^{t+1} &= \argmax_{\aug{X} \in \aug{\mc{X}}} L_{\beta_Y, \beta_Z}(\aug{X},Y^t,Z^t,\Lambda^t,M^t) \label{INT-ADMM X-update}\\
    Y^{t+1} &= \argmax_{Y \in \mathbb{S}^{n+1}_+} L_{\beta_Y, \beta_Z}(\aug{X}^{t+1},Y,Z^t,\Lambda^t,M^t) \label{INT-ADMM Y-update}\\
    Z^{t+1} &= \argmax_{Z \in \mc{Z}} L_{\beta_Y, \beta_Z}(\aug{X}^{t+1},Y^{t+1},Z,\Lambda^t,M^t) \label{INT-ADMM Z-update}\\
    \Lambda^{t+1} &= \Lambda^t + \beta_Y (\aug{X}^{t+1}-Y^{t+1}) \label{INT-ADMM L-update}\\
    M^{t+1} &= M^t + \beta_Z (\aug{X}^{t+1}-Z^{t+1}). \label{INT-ADMM M-update}
\end{align}
The maximization problems~\eqref{INT-ADMM X-update}-\eqref{INT-ADMM Z-update} can be rewritten as the following projection problems:
\begin{align}
    & \argmax_{\aug{X} \in \aug{\mc{X}}} L_{\beta_Y, \beta_Z}(\aug{X},Y^t,Z^t,\Lambda^t,M^t) = \mc{P}_{\mc{\aug{X}}}\left(\frac{1}{\beta_Y + \beta_Z}\left(\beta_Y Y^t + \beta_Z Z^t + \aug{I} - \Lambda^t - M^t\right)\right) \label{INT-ADMM X-projection}\\
    & \argmax_{Y \in \mathbb{S}^{n+1}_+} L_{\beta_Y, \beta_Z}(\aug{X}^{t+1},Y,Z^t,\Lambda^t,M^t) = \mc{P}_{ \mathbb{S}^{n+1}_+}\left(\aug{X}^{t+1} + \frac{1}{\beta_Y}\Lambda^t\right) \label{INT-ADMM Y-projection}\\
    & \argmax_{Z \in \mc{Z}} L_{\beta_Y, \beta_Z}(\aug{X}^{t+1},Y^{t+1},Z,\Lambda^t,M^t) = \mc{P}_{\mc{Z}}\left(\aug{X}^{t+1} + \frac{1}{\beta_Z}M^t\right). \label{INT-ADMM Z-projection}
\end{align}
The projections onto $\aug{\mc{X}}$ and $\mathbb{S}^{n+1}_+$ were already discussed in Section~\ref{Sec: ADMM scheme}.
The projection onto $\mc{Z}$ admits the following closed-form expression:
\begin{equation}
    \mc{P}_{\mc{Z}}(Z) = \frac{n+1}{2} \frac{Z-C}{\norm{Z-C}} + C.
\end{equation}

\subsection{Achieving convergence and escaping local optima} \label{Sec: Convergence INT-ADMM}
Given that $\mc{Z}$ is non-convex, the INT-ADMM is not guaranteed to converge, and if it converges, the obtained solution is not necessarily optimal for~\eqref{INT-ADMM split}. 
Preliminary experiments showed that the solution quality strongly depends on the choice of $\beta_Y$ and $\beta_Z$. 
As proposed in~\cite{Wu2019LpBoxADMM}, we set $\beta_Y = \beta_Z \eqdef \beta$, and we gradually increase $\beta$ over the course of the algorithm. 
We initialize $\beta$ to $\beta^0$ and then multiply by a factor $\beta_{incr} > 1$ after every iteration. 
Therefore, the emphasis initially lies on obtaining a high objective value, while over time the focus shifts to enforcing feasibility.
Using the parameter values specified in Table~\ref{tab:parameters INT-ADMM} in Section~\ref{Sec: Computational experiments}, our algorithm consistently converged on all benchmark instances to a solution satisfying the following criterion:
$$\max\left\{\frac{\norm{\aug{X}^t- Y^t}}{1 + \norm{\aug{X}^t}}, \frac{\norm{\aug{X}^t - Z^t}}{1 + \norm{\aug{X}^t}}\right\} \leq \eps_\texttt{INT-ADMM}.$$
Upon meeting this criterion, the matrix $\aug{X}^t$ (whose entries are close to integer) is rounded to a fully integer matrix, after which its feasibility is verified.

Our algorithm also contains a novel mechanism for escaping local optima. Namely, whenever the above convergence criterion is satisfied, we save the feasible solution found, and then multiply $\beta$ by a factor $\beta_{decr} \in (0,1)$ (taking care that $\beta$ does not drop below a minimum value $\beta_{min})$, thereby shifting the emphasis back to the objective value. 
The INT-ADMM then continues for at least $10$ more iterations until the convergence criterion is met again. 
This process is repeated until the algorithm has not improved the best solution so far for \texttt{maxTriesWithoutImpr} consecutive tries, or the current solution value matches a known upper bound.

\subsection{Reducing runtime}
Finally, we discuss some additional features of the INT-ADMM that are aimed to reduce runtime.

First, we initialize the INT-ADMM using the final solution $\aug{X}^*$ obtained from the CP-ADMM. 
Specifically, we set $\aug{X}^0 = \mc{P}_{\aug{\mc{X}}}(\aug{X}^*)$, $Y^0 = \mc{P}_{\mathbb{S}^{n+1}_+}(\aug{X}^*)$, $Z^0 = \mc{P}_{\mc{Z}}(\aug{X}^*)$, $\Lambda^0 = \beta_Y(\aug{X}^0 - Y^0)$ and $M^0 = \beta_Z(\aug{X}^0 - Z^0)$.
Furthermore, we experimented with additionally including in $\mc{Z}$ some of the equality constraints that are already in $\aug{\mc{X}}$.
Among the tested variants, only adding the constraint $Z_{00} = k$ proved beneficial.
The projection onto $\mc{Z}' \eqdef \{Z \in \mc{Z}: Z_{00} = k\}$ is given by 
\begin{equation}
    \mc{P}_{\mc{Z}'}(Z) = \frac{\sqrt{(n+1)^2-1}}{2} \frac{Z'-C}{\norm{Z'-C}} + C - \begin{pmatrix} 1/2 & 0_n^\top \\ 0_n & \textbf{0}_n \end{pmatrix}, \quad\text{where } Z' = Z + \begin{pmatrix} \frac{1}{2} + k -Z_{00} & 0_n^\top \\ 0_n & \textbf{0}_n \end{pmatrix}.
\end{equation}
Finally, observe that the $Y$- and $Z$-updates \eqref{INT-ADMM Y-projection} and~\eqref{INT-ADMM Z-projection} are independent of each other and could therefore be carried out in parallel. 
In our implementation, however, these updates are performed sequentially.

\section{Computational experiments} \label{Sec: Computational experiments}
This section presents the results of our computational experiments.
In Section~\ref{Sec: Computational setup}, we describe our computational setup.
In Section~\ref{Sec: Comparison with methods from the literature}, we compare the performance of our methods with methods from the literature, and we evaluate the importance of the mechanism to escape local optima in the INT-ADMM.
In Sections~\ref{Sec: Effect of valid inequalities} and~\ref{Sec: maxCutsPerVar}, we compare various configurations of the CP-ADMM, focusing, respectively, on the valid inequalities considered and on the effect of limiting the maximum number of new cuts per variable.
Finally, in Section~\ref{Sec: Chromatic number}, we use the upper bounds obtained with the CP-ADMM to derive lower bounds on the chromatic number of benchmark graphs.

\subsection{Computational setup} \label{Sec: Computational setup}
In Sections~\ref{Sec: Comparison with methods from the literature}--\ref{Sec: maxCutsPerVar}, we consider 68 benchmark instances used in \cite{Campelo2010ILPLagr} and \cite{Januschowski2011ILPSym}.
Specifically, we consider instances whose graphs originate from the Second DIMACS Implementation Challenge for the maximum clique problem \citep{Johnson1996DIMACS} or from the COLOR02 symposium \citep{Johnson2002COLOR02}.
All graphs in our tables that end with the letter ``c'' are complements of the original DIMACS graphs.
Many of these instances were also used in \cite{Kuryatnikova2022MkCSProblem}. 
We only consider instances that were not solved to optimality in that paper, since our upper bounds, when rounded down to the nearest integer, are at least as strong as the bounds reported there on all benchmark instances.

Our algorithms were implemented in MATLAB R2024b and the code is available upon request.
The experiments were executed on a virtual machine with an AMD EPYC-Rome processor, 2.0 GHz with 64 GB of RAM, running Ubuntu 22.04.5 LTS, and using a single CPU thread.
Linear programs were solved with Gurobi 13.00 using the parser CVX, version 2.2 \cite{CVX}.

Tables~\ref{tab:parameters CP-ADMM} and~\ref{tab:parameters INT-ADMM} report the parameter values used for the CP-ADMM and INT-ADMM, respectively.
For $\eps_\texttt{ADMM}$ and \texttt{maxInnerIter}, we present in parentheses the alternative values used in the final outer iteration of the CP-ADMM.

\begin{table}[H]
\centering
\caption{Parameter values for the CP-ADMM used in the computational experiments}
\label{tab:parameters CP-ADMM}
\begin{tabular}{@{}llllllll@{}}
\toprule
Parameter & Value &  & Parameter & Value &  & Parameter & Value \\ \cmidrule(r){1-2} \cmidrule(lr){4-5} \cmidrule(l){7-8} 
$\beta$ & $1.2$ &  & \texttt{minIneq} & $n/4$ &  & \texttt{timeLimitGlobal} & $3600$s \\
$\gamma$ & $1.617$ &  & \texttt{minIneqPhase1} & $n$ &  & \texttt{timeLimitCliques} & $10$s \\
$\eps_\texttt{ADMM}$ & $10^{-4}$ ($\to 10^{-5}$) &  & \texttt{maxIneq} & $5n$ &  & \texttt{timeLimitHoles} & $10$s \\
$\eps_{Dyk}$ & $10^{-2}$ &  & \texttt{maxCutsPerVar} & $5$ &  & \texttt{maxCliques} & $100{,}000$ \\
\texttt{maxInnerIter} & $2000$ ($\to 10{,}000$) &  & \texttt{minImpr} & $0.025$ &  & \texttt{maxCliquePairs} & $100{,}000$ \\
\texttt{minViol} & $10^{-2}$ &  & \texttt{minImprPhase1} & $0.25$ &  & \texttt{maxHoles} & $100{,}000$ \\ \bottomrule
\end{tabular}
\end{table}

\begin{table}[H]
\centering
\caption{Parameter values for the INT-ADMM used in the computational experiments.}
\label{tab:parameters INT-ADMM}
\begin{tabular}{@{}llllllll@{}}
\toprule
Parameter & Value &  & Parameter & Value &  & Parameter & Value \\ \cmidrule(r){1-2} \cmidrule(lr){4-5} \cmidrule(l){7-8} 
$\beta^0$ & $0.05$ &  & $\beta_{decr}$ & $0.5$ &  & $\eps_\texttt{INT-ADMM}$ & $10^{-3}$ \\
$\beta_{incr}$ & $1.0001$ &  & $\beta_{min}$ & 0.001 &  & \texttt{maxTriesWithoutImpr} & $3$ \\ \bottomrule
\end{tabular}
\end{table}

We also ran two heuristics from \cite{Kuryatnikova2022MkCSProblem} on our machine.
The first heuristic (HEUR1) reduces a M$k$CS instance with graph $\mc{G}$ to a maximum stable set instance defined on the Cartesian product of the complete graph on $k$ vertices and $\mc{G}$.
The resulting instance is then solved using a stable set heuristic.
The second heuristic (HEUR2) is a tabu search heuristic.
The maximum of the two obtained lower bounds is used for the optimality check in the stopping criteria of the CP-ADMM (see Section~\ref{Sec: Stopping criteria}).
Even though the INT-ADMM often provides better lower bounds than the two heuristics, its higher computational cost makes it less suitable for use in the stopping criteria of the CP-ADMM.
The upper bound obtained using the CP-ADMM is used for the optimality check in the INT-ADMM (see Section~\ref{Sec: Convergence INT-ADMM}).

\subsection{Comparison with methods from the literature} \label{Sec: Comparison with methods from the literature}
In Tables~\ref{Table: ResultsCC} and~\ref{Table: ResultsJP}, we compare the performance of our methods with methods from the literature on the instances considered in \cite{Campelo2010ILPLagr} and \cite{Januschowski2011ILPSym}, respectively.

In both tables, each row corresponds to an instance specified by the first four columns: the name of the graph, its number of vertices ($n$) and its density (given by $\rho = |\mc{E}|/\binom{n}{2}$), and the number of colors ($k$).
For the instances from \cite{Januschowski2011ILPSym}, Table~\ref{Table: ResultsJP} additionally reports the optimal values $\alpha_k(\mc{G})$ provided in that paper. 
No runtimes on individual instances are given there, but the authors report that their best method could prove optimality within about  one hour per instance. 
The next part of each table contains upper bounds, with corresponding runtimes (in seconds) in parentheses. 
For the instances from \cite{Campelo2010ILPLagr}, we first report the best upper bound obtained in that work, where a time limit of 1800 seconds was used. 
Runtimes are omitted here because of differences in hardware.
We further present the bound $\theta_k(\mc{G})$ as computed in \cite{Kuryatnikova2022MkCSProblem} using the solver Mosek. We also present the strongest bound from that paper, namely the bound $\theta^1_k(\mc{G})$ that is strengthened using BQP cuts. The authors only report results on instances where the solver did not run out of memory; the other entries in the tables are left blank.
The runtimes reported in that paper are included for reference, but should be interpreted with care as a different machine was used  for that work (a computer with two processors Intel\textsuperscript{\textregistered} Xeon\textsuperscript{\textregistered} Gold 6126 CPU @ 2.60 GHz and 512 GB of RAM).
The following columns present the bounds and runtimes of the standard ADMM without cuts (denoted by STD-ADMM) and the CP-ADMM, respectively.
Recall that the valid upper bounds computed by the STD-ADMM approximate $\theta_k(\mc{G})$.
As will be motivated in the next section, the CP-ADMM is configured here to use generalized clique and odd-hole inequalities, but no generalized triangle inequalities. 
Bounds preceded by a ``$\leq$''-sign indicate that a method was terminated early because the upper bound rounded down equaled a lower bound given by HEUR1 or HEUR2. 
Moreover, the CP-ADMM was not run on instances where the STD-ADMM already provided an optimal bound.
The final part of each table concerns the methods for finding lower bounds.
For the instances of \cite{Campelo2010ILPLagr}, we first report the best lower bound from that work.
Columns HEUR1, HEUR2 and INT-ADMM report the lower bounds and runtimes of the respective methods.
To evaluate the role of the mechanism to escape local optima in the INT-ADMM (see Section~\ref{Sec: Convergence INT-ADMM}), the columns INT-ADMM$^-$ report the objective value of the first feasible solution obtained by the algorithm, together with the corresponding runtime.
The best upper and lower bounds for each instance are depicted in bold.

Furthermore, Table~\ref{Table: Additional} reports for each instance the number of iterations required by each ADMM-based method. For the CP-ADMM, the table additionally specifies the number of outer iterations (in parentheses) and the number of added cuts.

The results can be summarized as follows:
\begin{itemize}
    \item The upper bounds computed by the STD-ADMM approximate well the bounds $\theta_k(\mc{G})$ computed in \cite{Kuryatnikova2022MkCSProblem} using Mosek, but the STD-ADMM typically requires much less runtime, and never ran out of memory.
    \item The CP-ADMM produces the strongest upper bounds on almost all instances, often improving upon the best previously known bound by multiple integers. 
    On the few instances where the CP-ADMM does not return the best bound, its bound after rounding down matches the best previously known integer bound.
    Furthermore, the CP-ADMM never ran out of memory and is often significantly faster than the IPM-based cutting plane algorithm of \cite{Kuryatnikova2022MkCSProblem} for finding the strengthened $\theta^1_k(\mc{G})$ bound, which is, on most instances, the best previously known upper bound. 
    There are also instances where the CP-ADMM does not improve upon the STD-ADMM, either because no violated inequalities are identified, or because the added inequalities do not lead to stronger bounds.
\end{itemize}

\begin{landscape}
\begin{table}[h]
\centering
\caption{Results on instances considered in \cite{Campelo2010ILPLagr}.}
\label{Table: ResultsCC}
\setlength{\tabcolsep}{4pt}
\resizebox{\linewidth}{!}{
\begin{tabular}{@{}lrrrrrrrrrrrrrrrrrrrrrrrrrrrrrrr@{}}
\toprule
\multicolumn{4}{c}{Instance} & \multicolumn{1}{c}{} & \multicolumn{13}{c}{Upper bounds} & \multicolumn{1}{c}{} & \multicolumn{13}{c}{Lower bounds} \\ \cmidrule(r){1-4} \cmidrule(lr){6-18} \cmidrule(l){20-32} 
\multicolumn{1}{c}{graph} & \multicolumn{1}{c}{$n$} & \multicolumn{1}{c}{$\rho$} & \multicolumn{1}{c}{$k$} & \multicolumn{1}{c}{} & \multicolumn{1}{c}{UB\cite{Campelo2010ILPLagr}} & \multicolumn{1}{c}{} & \multicolumn{2}{c}{$\theta_k(\mc{G})$} & \multicolumn{1}{c}{} & \multicolumn{2}{c}{$\theta^1_k(\mc{G})$+cuts} & \multicolumn{1}{c}{} & \multicolumn{2}{c}{STD-ADMM} & \multicolumn{1}{c}{} & \multicolumn{2}{c}{CP-ADMM} & \multicolumn{1}{c}{} & \multicolumn{1}{c}{LB\cite{Campelo2010ILPLagr}} & \multicolumn{1}{c}{} & \multicolumn{2}{c}{HEUR1} & \multicolumn{1}{c}{} & \multicolumn{2}{c}{HEUR2} & \multicolumn{1}{c}{} & \multicolumn{2}{c}{\makebox[2em][c]{INT-ADMM$^-$}} & \multicolumn{1}{c}{} & \multicolumn{2}{c}{INT-ADMM} \\ \midrule
C125.9c & 125 & 0.10 & 2 &  & 79.4 &  & 74.63 & (29) &  & 74.10 & (1130) &  & 74.64 & (1) &  & \textbf{70.46} & (262) &  & 61 &  & 61 & ($<$1) &  & \textbf{64} & (2) &  & 62 & (6) &  & \textbf{64} & (40) \\
 &  &  & 3 &  & 115.6 &  & 107.27 & (25) &  & 105.31 & (1806) &  & 107.28 & (1) &  & \textbf{102.29} & (241) &  & 81 &  & 83 & ($<$1) &  & 87 & (2) &  & 88 & (1) &  & \textbf{90} & (55) \\
keller4c & 171 & 0.35 & 2 &  & 27.9 &  & 26.93 & (62) &  & 26.93 & (1217) &  & 26.94 & (29) &  & \textbf{24.69} & (752) &  & 20 &  & 18 & ($<$1) &  & \textbf{22} & (4) &  & \textbf{22} & (6) &  & \textbf{22} & (70) \\
 &  &  & 3 &  & 41.9 &  & 40.40 & (68) &  & 40.40 & (1190) &  & 40.40 & (29) &  & \textbf{36.93} & (245) &  & 30 &  & 31 & ($<$1) &  & 32 & (4) &  & \textbf{33} & (5) &  & \textbf{33} & (77) \\
gen200\_p0.9\_44c & 200 & 0.10 & 2 &  & 88.0 &  & 88.00 & (258) &  & 87.99 & (13570) &  & 88.02 & (5) &  & \textbf{87.17} & (161) &  & 64 &  & 71 & ($<$1) &  & \textbf{82} & (6) &  & 77 & (3) &  & \textbf{82} & (103) \\
 &  &  & 3 &  & 132.0 &  & 131.94 & (258) &  & 131.66 & (16242) &  & 131.96 & (9) &  & \textbf{129.84} & (154) &  & 93 &  & 105 & ($<$1) &  & 113 & (6) &  & 112 & (3) &  & \textbf{116} & (205) \\
gen200\_p0.9\_55c & 200 & 0.10 & 2 &  & 109.0 &  & 100.84 & (338) &  & 100.36 & (14187) &  & 100.85 & (4) &  & \textbf{98.73} & (139) &  & 69 &  & 83 & ($<$1) &  & \textbf{94} & (6) &  & 90 & (3) &  & \textbf{94} & (99) \\
 &  &  & 3 &  & 161.6 &  & 146.22 & (312) &  & 145.34 & (10893) &  & 146.24 & (4) &  & \textbf{142.13} & (257) &  & 99 &  & 112 & ($<$1) &  & 127 & (6) &  & 127 & (4) &  & \textbf{129} & (179) \\
san200\_0.9\_2c & 200 & 0.10 & 2 &  & 117.0 &  & 106.61 & (353) &  & 106.03 & (13590) &  & 106.63 & (3) &  & \textbf{104.40} & (37) &  & 75 &  & 97 & ($<$1) &  & 99 & (6) &  & 99 & (4) &  & \textbf{101} & (65) \\
 &  &  & 3 &  & 184.0 &  & 152.84 & (328) &  & 151.16 & (11889) &  & 152.88 & (3) &  & \textbf{147.79} & (366) &  & 98 &  & 127 & ($<$1) &  & 131 & (4) &  & 137 & (3) &  & \textbf{138} & (87) \\
san200\_0.7\_2c & 200 & 0.30 & 2 &  & 36.0 &  & 35.63 & (220) &  & 35.60 & (12066) &  & 35.63 & (5) &  & \textbf{34.93} & (63) &  & 31 &  & 28 & ($<$1) &  & 30 & (5) &  & \textbf{34} & (3) &  & \textbf{34} & (3) \\
 &  &  & 3 &  & 54.0 &  & 53.24 & (190) &  & 53.22 & (5508) &  & 53.24 & (4) &  & \textbf{51.58} & (95) &  & 43 &  & 41 & ($<$1) &  & 42 & (5) &  & \textbf{49} & (3) &  & \textbf{49} & (69) \\
brock200\_4c & 200 & 0.34 & 2 &  & 51.8 &  & 42.24 & (115) &  & 42.24 & (1697) &  & 42.25 & (1) &  & \textbf{40.97} & (40) &  & 27 &  & 30 & ($<$1) &  & 30 & (5) &  & 29 & (13) &  & \textbf{31} & (137) \\
 &  &  & 3 &  & 76.6 &  & 63.36 & (117) &  & 63.36 & (1636) &  & 63.37 & (1) &  & \textbf{61.40} & (44) &  & 40 &  & 42 & ($<$1) &  & 45 & (5) &  & \textbf{46} & (16) &  & \textbf{46} & (106) \\
brock200\_2c & 200 & 0.50 & 2 &  & 30.1 &  & 28.26 & (60) &  & 28.26 & (1291) &  & 28.26 & (2) &  & \textbf{27.64} & (47) &  & 19 &  & 19 & ($<$1) &  & 20 & (4) &  & \textbf{21} & (22) &  & \textbf{21} & (114) \\
 &  &  & 3 &  & 44.9 &  & 42.39 & (61) &  & 42.39 & (1168) &  & 42.39 & (1) &  & \textbf{41.29} & (52) &  & 28 &  & 28 & ($<$1) &  & 29 & (4) &  & 30 & (27) &  & \textbf{31} & (142) \\
c-fat200-5c & 200 & 0.57 & 2 &  & \textbf{116.0} &  & 120.69 & (35) &  & 119.42 & (7649) &  & 120.73 & (24) &  & \textbf{$\leq$116.77} & (81) &  & \textbf{116} &  & \textbf{116} & ($<$1) &  & 115 & (11) &  & 114 & (32) &  & \textbf{116} & (107) \\
 &  &  & 3 &  & \textbf{172.0} &  & 181.03 & (33) &  & 175.24 & (6955) &  & 181.09 & (11) &  & \textbf{$\leq$172.88} & (49) &  & \textbf{172} &  & \textbf{172} & ($<$1) &  & \textbf{172} & (11) &  & 169 & (4) &  & \textbf{172} & (5) \\
C250.9c & 250 & 0.10 & 2 &  & 134.5 &  & 111.63 & (1952) &  &  &  &  & 111.64 & (5) &  & \textbf{109.04} & (94) &  & 74 &  & 81 & ($<$1) &  & 85 & (10) &  & 80 & (4) &  & \textbf{86} & (248) \\
 &  &  & 3 &  & 201.6 &  & 167.15 & (1861) &  &  &  &  & 167.22 & (3) &  & \textbf{163.58} & (42) &  & 106 &  & 116 & ($<$1) &  & 122 & (10) &  & 114 & (6) &  & \textbf{125} & (223) \\
p\_hat300-2c & 300 & 0.51 & 2 &  & 62.1 &  & 52.96 & (626) &  &  &  &  & 52.97 & (77) &  & \textbf{52.79} & (179) &  & 42 &  & 43 & ($<$1) &  & 45 & (11) &  & 45 & (6) &  & \textbf{46} & (207) \\
 &  &  & 3 &  & 91.0 &  & 77.28 & (718) &  &  &  &  & 77.29 & (48) &  & \textbf{77.09} & (186) &  & 55 &  & 58 & ($<$1) &  & 61 & (12) &  & \textbf{62} & (6) &  & \textbf{62} & (210) \\
p\_hat300-1c & 300 & 0.76 & 2 &  & 20.9 &  & \textbf{20.04} & (107) &  &  &  &  & \textbf{20.04} & (78) &  & \textbf{20.04} & (103) &  & 14 &  & 14 & ($<$1) &  & 14 & (11) &  & \textbf{16} & (32) &  & \textbf{16} & (196) \\
 &  &  & 3 &  & 31.0 &  & \textbf{30.06} & (109) &  &  &  &  & \textbf{30.06} & (78) &  & \textbf{30.06} & (103) &  & 20 &  & 21 & ($<$1) &  & 21 & (12) &  & 23 & (52) &  & \textbf{24} & (321) \\
MANN-a27c & 378 & 0.01 & 2 &  & \textbf{261.0} &  &  &  &  &  &  &  & 261.31 & (85) &  & 261.06 & (51) &  & 246 &  & 249 & ($<$1) &  & \textbf{250} & (22) &  & 191 & (38) &  & 240 & (451) \\
 &  &  & 3 &  & \textbf{378.0} &  &  &  &  &  &  &  & 378.12 & (4) &  & 378.12 & (5) &  & 368 &  & 370 & ($<$1) &  & \textbf{372} & (16) &  & 302 & (7) &  & 367 & (486) \\
san400-0-9-1c & 400 & 0.10 & 2 &  & 199.0 &  &  &  &  &  &  &  & 173.11 & (11) &  & \textbf{170.83} & (100) &  & 132 &  & 147 & ($<$1) &  & 108 & (26) &  & 151 & (12) &  & \textbf{155} & (442) \\
 &  &  & 3 &  & 289.9 &  &  &  &  &  &  &  & 245.77 & (15) &  & \textbf{240.69} & (352) &  & 163 &  & 202 & ($<$1) &  & 159 & (23) &  & 201 & (7) &  & \textbf{205} & (618) \\
c-fat500-10c & 500 & 0.63 & 2 &  & \textbf{252.0} &  &  &  &  &  &  &  & 252.19 & (143) &  & \textbf{} &  &  & \textbf{252} &  & \textbf{252} & ($<$1) &  & 251 & (85) &  & 250 & (292) &  & 250 & (839) \\
 &  &  & 3 &  & \textbf{376.0} &  &  &  &  &  &  &  & 376.42 & (143) &  & \textbf{} &  &  & \textbf{376} &  & \textbf{376} & ($<$1) &  & \textbf{376} & (69) &  & \textbf{376} & (22) &  & \textbf{376} & (22) \\ \bottomrule
\end{tabular}}
\end{table}

\begin{table}[h]
\centering
\caption{Results on instances considered in \cite{Januschowski2011ILPSym}.}
\label{Table: ResultsJP}
\setlength{\tabcolsep}{4pt}
\resizebox{0.93\linewidth}{!}{
\begin{tabular}{@{}lrrrrrrrrrrrrrrrrrrrrrrrrrrrr@{}}
\toprule
\multicolumn{5}{c}{Instance} & \multicolumn{1}{c}{} & \multicolumn{11}{c}{Upper   bounds} & \multicolumn{1}{c}{} & \multicolumn{11}{c}{Lower   bounds} \\ \cmidrule(r){1-5} \cmidrule(lr){7-17} \cmidrule(l){19-29} 
\multicolumn{1}{c}{graph} & \multicolumn{1}{c}{$n$} & \multicolumn{1}{c}{$\rho$} & \multicolumn{1}{c}{$k$} & \multicolumn{1}{c}{$\alpha_k(\mc{G})$} & \multicolumn{1}{c}{} & \multicolumn{2}{c}{$\theta_k(\mc{G})$} & \multicolumn{1}{c}{} & \multicolumn{2}{c}{$\theta^1_k(\mc{G})$+cuts} & \multicolumn{1}{c}{} & \multicolumn{2}{c}{STD-ADMM} & \multicolumn{1}{c}{} & \multicolumn{2}{c}{CP-ADMM} & \multicolumn{1}{c}{} & \multicolumn{2}{c}{HEUR1} & \multicolumn{1}{c}{} & \multicolumn{2}{c}{HEUR2} & \multicolumn{1}{c}{} & \multicolumn{2}{c}{\makebox[2em][c]{INT-ADMM$^-$}} & \multicolumn{1}{c}{} & \multicolumn{2}{c}{INT-ADMM} \\ \midrule
queen6\_6 & 36 & 0.46 & 6 & 32 &  & 35.84 & ($<$1) &  & 35.81 & (3) &  & 35.84 & ($<$1) &  & \textbf{33.60} & (38) &  & 30 & ($<$1) &  & \textbf{32} & (1) &  & \textbf{32} & (4) &  & \textbf{32} & (9) \\
myciel5 & 47 & 0.22 & 4 & 44 &  & \textbf{47.00} & ($<$1) &  & \textbf{47.00} & (4) &  & \textbf{47.00} & ($<$1) &  & \textbf{47.00} & (1) &  & \textbf{44} & ($<$1) &  & \textbf{44} & (1) &  & \textbf{44} & (4) &  & \textbf{44} & (13) \\
 &  &  & 5 & 46 &  & \textbf{47.00} & ($<$1) &  & \textbf{47.00} & (7) &  & \textbf{47.00} & ($<$1) &  & \textbf{47.00} & ($<$1) &  & \textbf{46} & ($<$1) &  & \textbf{46} & (5) &  & 44 & (6) &  & \textbf{46} & (17) \\
1-Insertions\_4 & 67 & 0.10 & 3 & 63 &  & \textbf{67.00} & (1) &  & \textbf{67.00} & (34) &  & \textbf{67.00} & ($<$1) &  & \textbf{67.00} & (1) &  & \textbf{63} & ($<$1) &  & \textbf{63} & (2) &  & \textbf{63} & (5) &  & \textbf{63} & (19) \\
1-FullIns\_4 & 93 & 0.14 & 3 & 87 &  & 92.59 & (11) &  & 91.33 & (554) &  & 92.60 & (2) &  & \textbf{89.83} & (178) &  & \textbf{87} & ($<$1) &  & \textbf{87} & (2) &  & \textbf{87} & (1) &  & \textbf{87} & (21) \\
myciel6 & 95 & 0.17 & 3 & 83 &  & 95.00 & (4) &  & 93.32 & (366) &  & 95.01 & ($<$1) &  & \textbf{89.59} & (336) &  & \textbf{83} & ($<$1) &  & \textbf{83} & (2) &  & 80 & (6) &  & 80 & (28) \\
4-FullIns\_3 & 114 & 0.08 & 3 & 106 &  & 107.40 & (20) &  & \textbf{107.25} & (901) &  & 107.41 & ($<$1) &  & 107.26 & (45) &  & 105 & ($<$1) &  & \textbf{106} & (3) &  & \textbf{106} & (3) &  & \textbf{106} & (29) \\
5-FullIns\_3 & 154 & 0.07 & 3 & 144 &  & 145.33 & (89) &  & 145.23 & (3551) &  & 145.35 & (1) &  & \textbf{145.22} & (49) &  & \textbf{144} & ($<$1) &  & \textbf{144} & (5) &  & \textbf{144} & (1) &  & \textbf{144} & (46) \\
c-fat200-1 & 200 & 0.08 & 10 & 180 &  & 184.67 & (300) &  & \textbf{184.65} & (5485) &  & 184.68 & (9) &  & 184.68 & (10) &  & \textbf{180} & ($<$1) &  & \textbf{180} & (6) &  & \textbf{180} & (16) &  & \textbf{180} & (98) \\
sanr200\_0.9 & 200 & 0.90 & 4 & 16 &  & \textbf{17.91} & (2) &  & \textbf{17.91} & (918) &  & \textbf{17.91} & (41) &  & \textbf{17.91} & (60) &  & \textbf{16} & ($<$1) &  & 15 & (5) &  & \textbf{16} & (26) &  & \textbf{16} & (108) \\
san200\_0.9\_2 & 200 & 0.90 & 4 & 16 &  & 17.21 & (3) &  & \textbf{17.20} & (2036) &  & 17.21 & (41) &  & 17.21 & (61) &  & \textbf{16} & ($<$1) &  & \textbf{16} & (5) &  & \textbf{16} & (19) &  & \textbf{16} & (106) \\
gen200\_p0.9\_55 & 200 & 0.90 & 4 & 17 &  & \textbf{18.15} & (4) &  & \textbf{18.15} & (1386) &  & \textbf{18.15} & (27) &  & \textbf{18.15} & (46) &  & 16 & ($<$1) &  & 15 & (5) &  & \textbf{17} & (18) &  & \textbf{17} & (99) \\
2-FullIns\_4 & 212 & 0.07 & 3 & 202 &  & 207.65 & (993) &  &  &  &  & 207.69 & (8) &  & \textbf{206.50} & (936) &  & 201 & ($<$1) &  & \textbf{202} & (6) &  & 189 & (2) &  & \textbf{202} & (106) \\
DSJC250.9 & 250 & 0.90 & 4 & 18 &  & \textbf{19.72} & (15) &  &  &  &  & 19.73 & (59) &  & 19.73 & (84) &  & 17 & ($<$1) &  & 16 & (8) &  & \textbf{18} & (34) &  & \textbf{18} & (157) \\
3-FullIns\_4 & 405 & 0.04 & 5 & 402 &  &  &  &  &  &  &  & 405.03 & (24) &  & \textbf{404.29} & (103) &  & \textbf{402} & ($<$1) &  & \textbf{402} & (21) &  & 377 & (8) &  & \textbf{402} & (248) \\
 &  &  & 6 & 404 &  &  &  &  &  &  &  & \textbf{405.05} & (6) &  & 405.06 & (90) &  & \textbf{404} & (1) &  & \textbf{404} & (22) &  & 386 & (5) &  & \textbf{404} & (280) \\
DSJR500.1 & 500 & 0.03 & 8 & 459 &  &  &  &  &  &  &  & 459.81 & (4) &  & \textbf{459.55} & (21) &  & 445 & (2) &  & 438 & (26) &  & 403 & (16) &  & \textbf{458} & (667) \\
 &  &  & 9 & 477 &  &  &  &  &  &  &  & 477.82 & (4) &  & \textbf{477.54} & (19) &  & 466 & (2) &  & 459 & (25) &  & 441 & (11) &  & \textbf{477} & (95) \\
 &  &  & 10 & 489 &  &  &  &  &  &  &  & 489.84 & (6) &  & \textbf{489.55} & (18) &  & 480 & (3) &  & 477 & (27) &  & 474 & (14) &  & \textbf{489} & (165) \\
 &  &  & 11 & 496 &  &  &  &  &  &  &  & 496.81 & (4) &  & \textbf{496.55} & (15) &  & 489 & (3) &  & 493 & (35) &  & 492 & (12) &  & \textbf{496} & (32) \\
c-fat500-2 & 500 & 0.07 & 15 & 300 &  &  &  &  &  &  &  & \textbf{300.06} & (3) &  & \textbf{} &  &  & \textbf{300} & (3) &  & \textbf{300} & (55) &  & \textbf{300} & (32) &  & \textbf{300} & (32) \\
 &  &  & 17 & 340 &  &  &  &  &  &  &  & \textbf{340.10} & (3) &  & \textbf{} &  &  & \textbf{340} & (5) &  & \textbf{340} & (45) &  & \textbf{340} & (46) &  & \textbf{340} & (46) \\
 &  &  & 20 & 400 &  &  &  &  &  &  &  & \textbf{400.16} & (3) &  & \textbf{} &  &  & \textbf{400} & (7) &  & \textbf{400} & (56) &  & \textbf{400} & (33) &  & \textbf{400} & (33) \\
 &  &  & 22 & 440 &  &  &  &  &  &  &  & \textbf{440.19} & (4) &  & \textbf{} &  &  & \textbf{440} & (10) &  & \textbf{440} & (51) &  & 433 & (76) &  & \textbf{440} & (98) \\
 &  &  & 25 & 490 &  &  &  &  &  &  &  & \textbf{$\leq$490.15} & (35) &  & \textbf{} &  &  & \textbf{490} & (13) &  & \textbf{490} & (47) &  & \textbf{490} & (85) &  & \textbf{490} & (85) \\
c-fat500-10 & 500 & 0.37 & 10 & 40 &  &  &  &  &  &  &  & \textbf{$\leq$40.09} & (36) &  & \textbf{} &  &  & \textbf{40} & ($<$1) &  & \textbf{40} & (48) &  & \textbf{40} & (163) &  & \textbf{40} & (163) \\
 &  &  & 15 & 60 &  &  &  &  &  &  &  & \textbf{$\leq$60.05} & (36) &  & \textbf{} &  &  & \textbf{60} & (1) &  & \textbf{60} & (50) &  & \textbf{60} & (159) &  & \textbf{60} & (159) \\
 &  &  & 17 & 68 &  &  &  &  &  &  &  & \textbf{$\leq$68.06} & (36) &  & \textbf{} &  &  & \textbf{68} & (2) &  & \textbf{68} & (50) &  & \textbf{68} & (151) &  & \textbf{68} & (151) \\
 &  &  & 20 & 80 &  &  &  &  &  &  &  & \textbf{80.04} & (34) &  & \textbf{} &  &  & \textbf{80} & (3) &  & \textbf{80} & (58) &  & \textbf{80} & (136) &  & \textbf{80} & (136) \\
 &  &  & 22 & 88 &  &  &  &  &  &  &  & \textbf{88.05} & (32) &  & \textbf{} &  &  & \textbf{88} & (3) &  & \textbf{88} & (60) &  & \textbf{88} & (151) &  & \textbf{88} & (151) \\
 &  &  & 25 & 100 &  &  &  &  &  &  &  & \textbf{100.05} & (28) &  & \textbf{} &  &  & \textbf{100} & (5) &  & \textbf{100} & (56) &  & \textbf{100} & (133) &  & \textbf{100} & (133) \\
4-FullIns\_4 & 690 & 0.03 & 5 & 684 &  &  &  &  &  &  &  & 686.59 & (326) &  & \textbf{686.38} & (3167) &  & \textbf{684} & (2) &  & \textbf{684} & (50) &  & 643 & (84) &  & \textbf{684} & (970) \\
 &  &  & 6 & 687 &  &  &  &  &  &  &  & 690.24 & (329) &  & \textbf{689.63} & (2300) &  & \textbf{687} & (3) &  & \textbf{687} & (42) &  & 625 & (31) &  & \textbf{687} & (669) \\
 &  &  & 7 & 689 &  &  &  &  &  &  &  & \textbf{690.23} & (347) &  & \textbf{690.23} & (515) &  & \textbf{689} & (3) &  & \textbf{689} & (38) &  & 659 & (50) &  & \textbf{689} & (1056) \\
5-FullIns\_4 & 1085 & 0.02 & 5 & 1076 &  &  &  &  &  &  &  & 1078.80 & (853) &  & \textbf{1078.52} & (3600) &  & \textbf{1076} & (6) &  & \textbf{1076} & (113) &  & 823 & (190) &  & \textbf{1076} & (3011) \\
 &  &  & 6 & 1079 &  &  &  &  &  &  &  & \textbf{1082.51} & (891) &  & 1082.64 & (3278) &  & \textbf{1079} & (10) &  & \textbf{1079} & (96) &  & 946 & (165) &  & \textbf{1079} & (3312) \\
 &  &  & 7 & 1082 &  &  &  &  &  &  &  & 1085.96 & (881) &  & \textbf{1085.52} & (3491) &  & \textbf{1082} & (12) &  & \textbf{1082} & (90) &  & 814 & (109) &  & \textbf{1082} & (2869) \\
 &  &  & 8 & 1084 &  &  &  &  &  &  &  & \textbf{1085.68} & (917) &  & \textbf{1085.68} & (937) &  & \textbf{1084} & (17) &  & \textbf{1084} & (97) &  & 971 & (131) &  & 1082 & (2865) \\ \bottomrule
\end{tabular}}
\end{table}
\end{landscape}

\begin{landscape}
\begin{table}[h]
\centering
\caption{Additional results on the number of iterations and cuts.}
\label{Table: Additional}
\setlength{\tabcolsep}{4pt}
\resizebox{0.95\linewidth}{!}{%
\begin{tabular}{@{}lrrrrrrrrrrrr@{\hspace{1cm}}lrrrrrrrrrrr@{}}
\toprule
\multicolumn{12}{c}{Instances considered in \cite{Campelo2010ILPLagr}} & \multicolumn{1}{c}{} & \multicolumn{12}{c}{Instances considered in \cite{Januschowski2011ILPSym}} \\ \cmidrule(r){1-12} \cmidrule(l){14-25} 
\multicolumn{2}{c}{Instance} & \multicolumn{1}{c}{} & \multicolumn{1}{c}{STD-ADMM} & \multicolumn{1}{c}{} & \multicolumn{3}{c}{CP-ADMM} & \multicolumn{1}{c}{} & \multicolumn{1}{c}{INT-ADMM$^-$} & \multicolumn{1}{c}{} & \multicolumn{1}{c}{INT-ADMM} & \multicolumn{1}{c}{} & \multicolumn{2}{c}{Instance} & \multicolumn{1}{c}{} & \multicolumn{1}{c}{STD-ADMM} & \multicolumn{1}{c}{} & \multicolumn{3}{c}{CP-ADMM} & \multicolumn{1}{c}{} & \multicolumn{1}{c}{INT-ADMM$^-$} & \multicolumn{1}{c}{} & \multicolumn{1}{c}{INT-ADMM} \\ \cmidrule(r){1-2} \cmidrule(lr){4-4} \cmidrule(lr){6-8} \cmidrule(lr){10-10} \cmidrule(lr){12-12} \cmidrule(lr){14-15} \cmidrule(lr){17-17} \cmidrule(lr){19-21} \cmidrule(lr){23-23} \cmidrule(l){25-25} 
\multicolumn{1}{c}{graph} & \multicolumn{1}{c}{$k$} & \multicolumn{1}{c}{} & \multicolumn{1}{c}{\#iterations} & \multicolumn{1}{c}{} & \multicolumn{2}{c}{\#iterations} & \multicolumn{1}{c}{\#cuts} & \multicolumn{1}{c}{} & \multicolumn{1}{c}{\#iterations} & \multicolumn{1}{c}{} & \multicolumn{1}{c}{\#iterations} & \multicolumn{1}{c}{} & \multicolumn{1}{c}{graph} & \multicolumn{1}{c}{$k$} & \multicolumn{1}{c}{} & \multicolumn{1}{c}{\#iterations} & \multicolumn{1}{c}{} & \multicolumn{2}{c}{\#iterations} & \multicolumn{1}{c}{\#cuts} & \multicolumn{1}{c}{} & \multicolumn{1}{c}{\#iterations} & \multicolumn{1}{c}{} & \multicolumn{1}{c}{\#iterations} \\ \midrule
C125.9c & 2 &  & 501 &  & 3008 & (13) & 3395 &  & 3266 &  & 22973 &  & queen6\_6 & 6 &  & 680 &  & 5569 & (12) & 627 &  & 13460 &  & 33684 \\
 & 3 &  & 548 &  & 2845 & (15) & 3411 &  & 750 &  & 31325 &  & myciel5 & 4 &  & 34 &  & 34 & (2) & 51 &  & 12441 &  & 38842 \\
keller4c & 2 &  & 10000 &  & 12625 & (9) & 3608 &  & 1792 &  & 22029 &  &  & 5 &  & 28 &  & 28 & (1) & 0 &  & 17687 &  & 50207 \\
 & 3 &  & 10000 &  & 3583 & (11) & 4487 &  & 1522 &  & 24570 &  & 1-Insertions\_4 & 3 &  & 54 &  & 75 & (3) & 410 &  & 5724 &  & 25286 \\
gen200\_p0.9\_44c & 2 &  & 1284 &  & 4779 & (9) & 4092 &  & 614 &  & 25324 &  & 1-FullIns\_4 & 3 &  & 1916 &  & 5994 & (11) & 3928 &  & 516 &  & 19638 \\
 & 3 &  & 2581 &  & 2932 & (10) & 5017 &  & 624 &  & 50164 &  & myciel6 & 3 &  & 99 &  & 6689 & (13) & 5608 &  & 5455 &  & 26107 \\
gen200\_p0.9\_55c & 2 &  & 1049 &  & 2258 & (10) & 3945 &  & 644 &  & 24090 &  & 4-FullIns\_3 & 3 &  & 304 &  & 2946 & (6) & 898 &  & 2058 &  & 19497 \\
 & 3 &  & 1144 &  & 2552 & (12) & 4785 &  & 1000 &  & 43577 &  & 5-FullIns\_3 & 3 &  & 487 &  & 2169 & (6) & 1058 &  & 347 &  & 17594 \\
san200\_0.9\_2c & 2 &  & 750 &  & 1492 & (6) & 3217 &  & 1007 &  & 15809 &  & c-fat200-1 & 10 &  & 2439 &  & 2439 & (1) & 0 &  & 4060 &  & 24600 \\
 & 3 &  & 925 &  & 2800 & (15) & 5335 &  & 660 &  & 21110 &  & sanr200\_0.9 & 4 &  & 10000 &  & 10000 & (1) & 0 &  & 6342 &  & 26154 \\
san200\_0.7\_2c & 2 &  & 1235 &  & 2208 & (6) & 2185 &  & 602 &  & 602 &  & san200\_0.9\_2 & 4 &  & 10000 &  & 10000 & (1) & 0 &  & 4586 &  & 25966 \\
 & 3 &  & 995 &  & 1957 & (9) & 4185 &  & 667 &  & 16397 &  & gen200\_p0.9\_55 & 4 &  & 6659 &  & 6659 & (1) & 0 &  & 4317 &  & 24179 \\
brock200\_4c & 2 &  & 137 &  & 622 & (6) & 2668 &  & 2940 &  & 32005 &  & 2-FullIns\_4 & 3 &  & 1967 &  & 10987 & (7) & 5288 &  & 381 &  & 24844 \\
 & 3 &  & 118 &  & 521 & (7) & 3274 &  & 3722 &  & 24646 &  & DSJC250.9 & 4 &  & 10000 &  & 10000 & (1) & 0 &  & 5908 &  & 27719 \\
brock200\_2c & 2 &  & 469 &  & 1090 & (5) & 998 &  & 4989 &  & 26193 &  & 3-FullIns\_4 & 5 &  & 2002 &  & 1400 & (4) & 1551 &  & 638 &  & 20615 \\
 & 3 &  & 304 &  & 883 & (6) & 1436 &  & 6205 &  & 32366 &  &  & 6 &  & 433 &  & 3903 & (2) & 178 &  & 371 &  & 23060 \\
c-fat200-5c & 2 &  & 7387 &  & 7970 & (5) & 3255 &  & 8075 &  & 27039 &  & DSJR500.1 & 8 &  & 196 &  & 248 & (3) & 817 &  & 940 &  & 39031 \\
 & 3 &  & 3566 &  & 3298 & (6) & 4443 &  & 1122 &  & 1273 &  &  & 9 &  & 227 &  & 232 & (3) & 476 &  & 639 &  & 5464 \\
C250.9c & 2 &  & 1066 &  & 2142 & (8) & 3420 &  & 691 &  & 42300 &  &  & 10 &  & 301 &  & 266 & (2) & 258 &  & 784 &  & 9332 \\
 & 3 &  & 644 &  & 1309 & (5) & 3217 &  & 1069 &  & 38292 &  &  & 11 &  & 209 &  & 252 & (2) & 166 &  & 667 &  & 1811 \\
p\_hat300-2c & 2 &  & 10000 &  & 10634 & (2) & 201 &  & 718 &  & 26862 &  & c-fat500-2 & 15 &  & 153 &  &  &  &  &  & 1835 &  & 1835 \\
 & 3 &  & 6219 &  & 10478 & (2) & 175 &  & 797 &  & 26903 &  &  & 17 &  & 157 &  &  &  &  &  & 2508 &  & 2508 \\
p\_hat300-1c & 2 &  & 10000 &  & 10000 & (1) & 0 &  & 4045 &  & 25305 &  &  & 20 &  & 177 &  &  &  &  &  & 1846 &  & 1846 \\
 & 3 &  & 10000 &  & 10000 & (1) & 0 &  & 6725 &  & 41314 &  &  & 22 &  & 207 &  &  &  &  &  & 4333 &  & 5608 \\
MANN-a27c & 2 &  & 8394 &  & 2875 & (3) & 559 &  & 3479 &  & 41660 &  &  & 25 &  & 2000 &  &  &  &  &  & 5033 &  & 5033 \\
 & 3 &  & 395 &  & 395 & (1) & 0 &  & 675 &  & 43893 &  & c-fat500-10 & 10 &  & 2000 &  &  &  &  &  & 9396 &  & 9396 \\
san400-0-9-1c & 2 &  & 984 &  & 1551 & (5) & 4511 &  & 991 &  & 35926 &  &  & 15 &  & 2000 &  &  &  &  &  & 9057 &  & 9057 \\
 & 3 &  & 1323 &  & 2607 & (12) & 7296 &  & 597 &  & 49768 &  &  & 17 &  & 2000 &  &  &  &  &  & 8439 &  & 8439 \\
c-fat500-10c & 2 &  & 10000 &  &  &  &  &  & 17398 &  & 48553 &  &  & 20 &  & 1927 &  &  &  &  &  & 7912 &  & 7912 \\
 & 3 &  & 10000 &  &  &  &  &  & 1399 &  & 1399 &  &  & 22 &  & 1758 &  &  &  &  &  & 8754 &  & 8754 \\
 &  &  &  &  &  &  &  &  &  &  &  &  &  & 25 &  & 1542 &  &  &  &  &  & 7773 &  & 7773 \\
 &  &  &  &  &  &  &  &  &  &  &  &  & 4-FullIns\_4 & 5 &  & 10000 &  & 10638 & (5) & 3899 &  & 2801 &  & 32419 \\
 &  &  &  &  &  &  &  &  &  &  &  &  &  & 6 &  & 10000 &  & 10566 & (4) & 3823 &  & 1025 &  & 21502 \\
 &  &  &  &  &  &  &  &  &  &  &  &  &  & 7 &  & 10000 &  & 10079 & (2) & 207 &  & 1627 &  & 33925 \\
 &  &  &  &  &  &  &  &  &  &  &  &  & 5-FullIns\_4 & 5 &  & 10000 &  & 9933 & (3) & 1542 &  & 2449 &  & 39010 \\
 &  &  &  &  &  &  &  &  &  &  &  &  &  & 6 &  & 10000 &  & 10306 & (3) & 981 &  & 2076 &  & 42540 \\
 &  &  &  &  &  &  &  &  &  &  &  &  &  & 7 &  & 10000 &  & 10349 & (3) & 1371 &  & 1395 &  & 37081 \\
 &  &  &  &  &  &  &  &  &  &  &  &  &  & 8 &  & 10000 &  & 10000 & (1) & 0 &  & 1676 &  & 36490 \\ \bottomrule
\end{tabular}}
\end{table}
\end{landscape}

\begin{itemize}
    \item The INT-ADMM always returned a high-quality feasible solution. For the instances considered in \cite{Campelo2010ILPLagr}, the INT-ADMM improves upon the best previously known lower bound on 20 instances, matches the best known bound on 7 instances (with the solution being optimal in at least 3 cases), and is worse on only 3 instances.
    For the instances considered in \cite{Januschowski2011ILPSym}, the INT-ADMM finds an optimal solution for all instances except myciel6 with $k = 3$ and 5-FullIns\_4 with $k=8$ (where HEUR1 and HEUR2 both find an optimal solution), and  DSJR500.1 with $k = 8$ (where the INT-ADMM improves upon HEUR1 and HEUR2).
    \item The mechanism to escape local optima in the INT-ADMM helps to find significantly improved solutions, but the mechanism is computationally costly as it requires a large number of iterations.
\end{itemize}

\subsection{Comparing the effectiveness of valid inequalities} \label{Sec: Effect of valid inequalities}
In this section, we compare several configurations of the CP-ADMM that include inequalities \eqref{eq T1}, \eqref{eq T2}, \eqref{eq cliqueExternal1}, \eqref{eq cliqueUnion} and/or \eqref{eq HoleExternal1}.
For the results presented in this subsection only, each configuration was run with a time limit of 900 seconds.
We restrict attention to 36 out of the 68 benchmark instances for which (i) at least one configuration yields a stronger bound than $\theta_k(\mc{G})$, (ii) none of the configurations attains optimality, and (iii) the majority of the configurations terminate within the time limit.

Figures~\ref{Fig: boxplotsBounds} and~\ref{Fig: boxplotsRuntimes} summarize the results using box and whisker plots. 
Each box extends from the first quartile to the third quartile of the data, with a solid line marking the median. Whiskers extend to the farthest data point lying within 1.5 times the inter-quartile range from the box. 
Outliers are shown by individual markers, and means are indicated by dashed lines.

Figure~\ref{Fig: boxplotsBounds} shows the effect of our valid inequalities on the obtained upper bounds. 
The plot on the left compares upper bounds obtained by configurations in which inequalities of a single type are added to \eqref{SDP}. 
The right plot focuses on the marginal benefit of various configurations that include generalized clique inequalities~\eqref{eq cliqueExternal1} on top of the configuration that only includes those inequalities.
Figure~\ref{Fig: boxplotsRuntimes} concerns the runtimes of each of the configurations. 
\begin{figure}[H]
    \centering
    \caption{Effect of valid inequalities on the obtained upper bounds.}
    \label{Fig: boxplotsBounds}
    \includegraphics[width=1\linewidth]{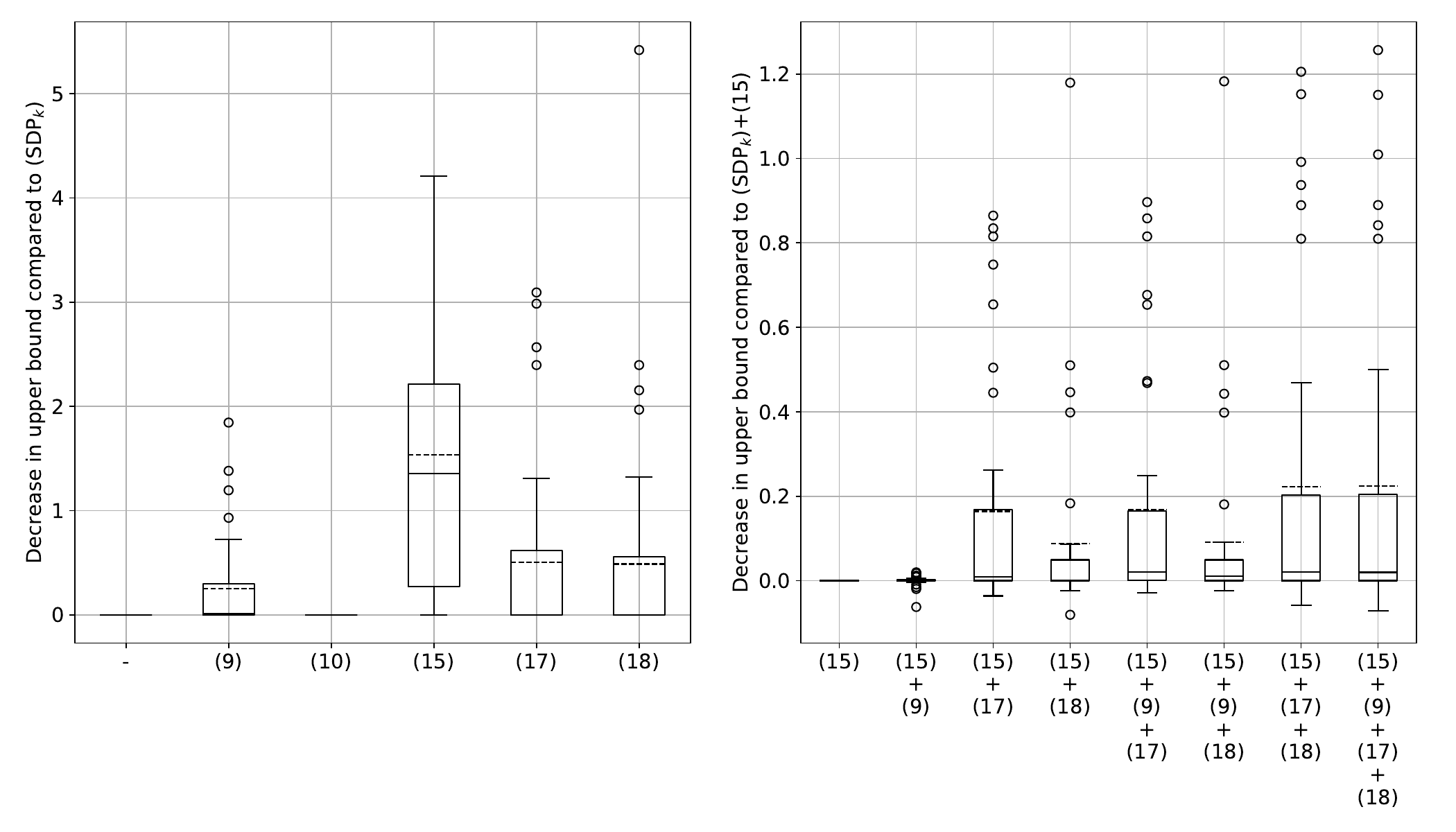}
\end{figure}
\vspace{-1cm}
\begin{figure}[H]
    \centering
    \caption{Required runtime per configuration.}
    \label{Fig: boxplotsRuntimes}
    \includegraphics[width=0.8\linewidth]{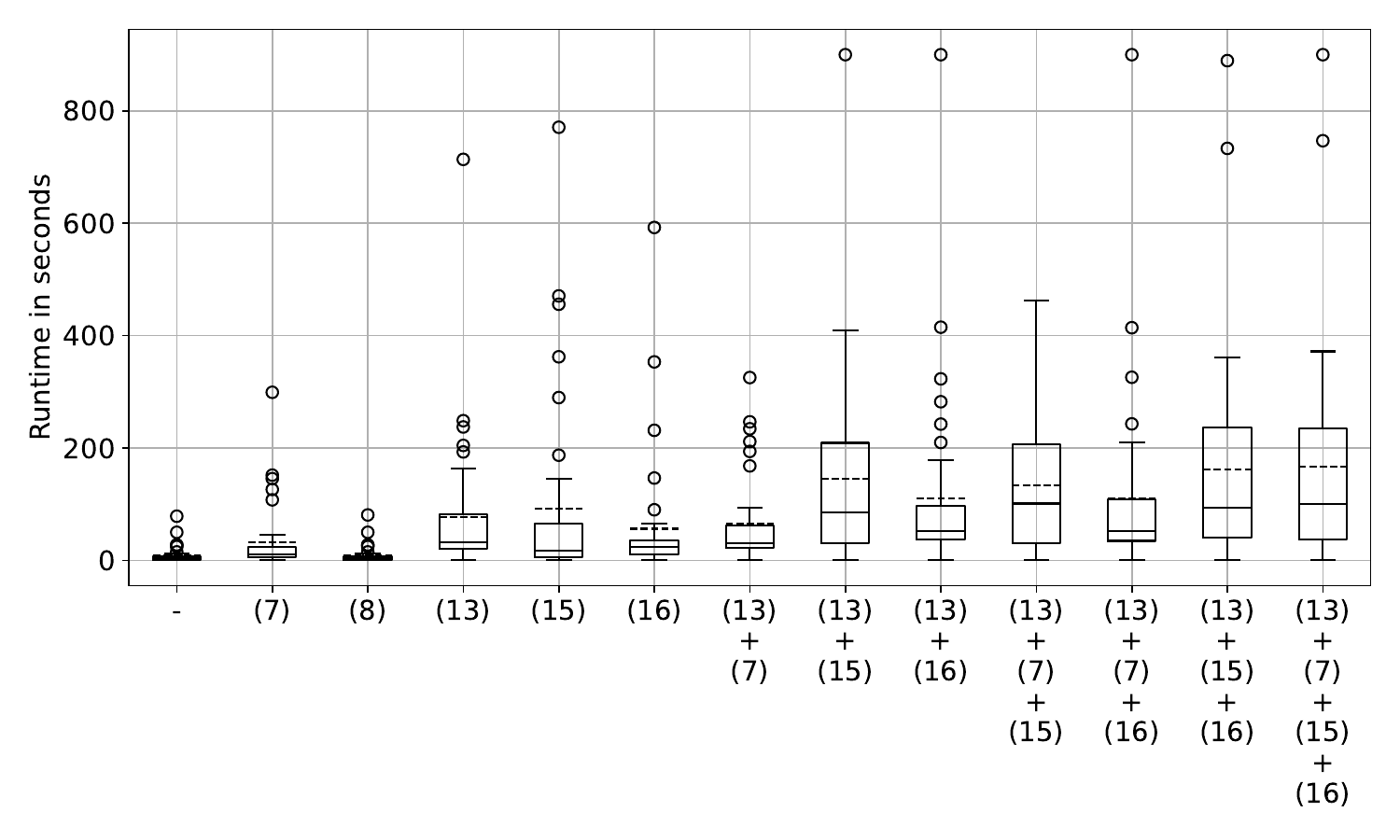}
\end{figure}
\vspace{-0.5cm}

When adding inequalities of a single type to \eqref{SDP}, the generalized clique inequalities~\eqref{eq cliqueExternal1} typically yield the best bounds by a considerable margin. 
On the other hand, generalized triangle inequalities~\eqref{eq T2} do not provide any improvement in bounds when tested on 11 instances for which $k \leq 2$. 
Recall that those inequalities are redundant for  $k>2$.
Conversely, triangle inequalities~\eqref{eq T1} do yield improved bounds.
Better bounds are typically obtained by adding generalized clique inequalities~\eqref{eq cliqueUnion} or generalized odd-hole inequalities~\eqref{eq HoleExternal1}, with their relative effectiveness varying across the instances.
In particular, inequalities \eqref{eq HoleExternal1} are highly effective for the instance with graph myciel6, which does not contain any cliques of size $3$ or larger (see the outlier on the top-right of the first plot). 
This instance is excluded from the second plot, as it would otherwise dominate the scale of the y-axis.
In terms of computational cost, the CP-ADMM typically requires most runtime to solve the relaxations with generalized clique inequalities~\eqref{eq cliqueExternal1} and~\eqref{eq cliqueUnion}, followed by the relaxation with generalized odd-hole inequalities~\eqref{eq HoleExternal1}.

Considering the marginal benefit of various configurations that include~\eqref{eq cliqueExternal1} on top of the configuration that only includes~\eqref{eq cliqueExternal1}, we observe the following.
The largest marginal benefits are obtained using generalized clique inequalities~\eqref{eq cliqueUnion}, but these bound improvements come at a significant computational cost. 
The generalized odd-hole inequalities~\eqref{eq HoleExternal1} also provide significant marginal improvements for many instances, but at a lower computational cost.
Despite being effective on their own, the generalized triangle inequalities~\eqref{eq T1} only lead to limited improvements on top of the other configurations that include inequalities~\eqref{eq cliqueExternal1}. 
This observation is in line with the computational results of \cite{Pucher2025PracticalExperience}, where the M$k$CS problem is considered for $k=1$.
Finally, we observe that considering more inequality types can sometimes lead to slightly weaker bounds. 
In these cases, the CP-ADMM is terminated before all violated inequalities are added to the relaxation, see Section~\ref{Sec: Stopping criteria}.

\subsection{Limiting the number of new cuts per variable} \label{Sec: maxCutsPerVar}
In this section, we report on the considerable benefit of limiting the maximum number of new cuts in which each variable may appear. 
Recall that in our default configuration of the CP-ADMM, \texttt{maxCutsPerVar} is set to $5$.
On the 45 instances where the default configuration added at least one cut, we also ran a version of the algorithm without the \texttt{maxCutsPerVar} limit.
On average, the default configuration requires only 50\% the runtime, 73\% the number of cuts and 42\% the amount of time per inner iteration compared to the configuration without the limit. 
One explanation for the decrease in runtime is that the \texttt{maxCutsPerVar} limit makes it easier to partition the cuts into a small number of clusters of non-overlapping cuts, thereby reducing the runtime required for Dykstra's cyclic projection algorithm.
On instances where the configuration without the \texttt{maxCutsPerVar} limit reaches the time limit, the default configuration often yields stronger upper bounds, as it can perform more iterations in the same amount of time.

\subsection{Lower bounds on the chromatic number} \label{Sec: Chromatic number}
\citet{Kuryatnikova2022MkCSProblem} exploited upper bounds for the M$k$CS problem to compute lower bounds for the chromatic number $\chi(\mc{G})$ of a graph. 
In particular, they used the following relation
\begin{equation}
    \chi(\mc{G}) \geq \max \left \{ k: UB_k(\mc{G}) < n \right \} + 1,
\end{equation}
where $UB_k(\mc{G})$ is an upper bound on $\alpha_k(\mc{G})$ for $k \in [n-1]$.
In this section, we show that our improved upper bounds for the M$k$CS problem yield improved lower bounds on $\chi(\mc{G})$.

To compute the required upper bounds, we use a slightly modified version of the CP-ADMM.
Namely, we adapt its stopping criteria by setting \texttt{minIneq} to $1$ and by not requiring any longer a minimum improvement \texttt{minImpr} in the valid upper bound per outer iteration.
Moreover, the CP-ADMM is terminated early whenever a valid upper bound strictly smaller than $n$ is found. 
We additionally compute a valid upper bound after every 100 inner iterations during the first outer iteration.

A straightforward approach would be to sequentially call the CP-ADMM for $k=1, 2, \hdots$ until reaching a value of $k$ for which $UB_k(\mc{G}) \geq n$. 
Instead, we employ a more efficient search strategy that reduces the number of calls to the CP-ADMM. Namely, we exploit the following relation:
\begin{equation}
    \chi(\mc{G}) \geq \ceil{\frac{kn}{\floor{UB_k(\mc{G})}}} \geq k + 1 \quad\forall k \in [n-1] \text{ with } UB_k(\mc{G}) < n.
\end{equation}
Concretely, we proceed iteratively over values of $k$, starting from $k=1$. 
For each considered value of~$k$, we compute an upper bound $UB_k(\mc{G})$ on $\alpha_k(\mc{G})$ using the CP-ADMM. If $UB_k(\mc{G}) < n$, we update $k$ to $\ceil{\frac{kn}{\floor{ UB_k(\mc{G})}}}$ and continue the algorithm. 
Otherwise, the algorithm terminates and the final value of $k$ is a valid lower bound on $\chi(\mc{G})$.
We also tested a binary search variant of the algorithm, but it performed worse in terms of runtime on average. 

We evaluate the obtained lower bounds on 36 benchmark graphs, consisting of 21 graphs from the COLOR02 symposium \citep{Johnson2002COLOR02} that were considered in \cite{Dukanovic2007SDPGraphColoring}, 3 further COLOR02 graphs considered in \cite{Gaar2020SDPGraphColoring}, and 3 Queen graphs and 9 random graphs considered in \cite{Pucher2025PracticalExperience}. 
We compare our bounds against SDP-based bounds reported in \cite{Dukanovic2007SDPGraphColoring, Gaar2020SDPGraphColoring, Kuryatnikova2022MkCSProblem, Pucher2025PracticalExperience}.

Our results are as follows.
On 3 of the 36 graphs, our algorithm improved upon the best previously known bounds:
\begin{itemize}
    \item On graph random\_100\_0.50 (with $n=100$ and $\rho = 0.50$), we obtain a lower bound of $12$, improving upon the lower bound $11$ found in \cite{Pucher2025PracticalExperience}.
    \item On graph DSJC125.5 (with $n=125$ and $\rho = 0.50$), we obtain a lower bound of $14$, improving upon the lower bound $12$ found in \cite{Dukanovic2007SDPGraphColoring} and \cite{Kuryatnikova2022MkCSProblem}.
    \item On graph DSJC125.9 (with $n=125$ and $\rho = 0.90$), we obtain a lower bound of $43$, improving upon the lower bound $38$ found in \cite{Dukanovic2007SDPGraphColoring} and \cite{Kuryatnikova2022MkCSProblem}.
\end{itemize}
On the graph myciel3 (with $n=20$ and $\rho = 0.36$), we obtain a lower bound of $3$, whereas the algorithm from \cite{Gaar2020SDPGraphColoring} yields a lower bound of $4$.
On the remaining 32 instances, our algorithm matches the best known lower bound.
Our algorithm required less than 3 minutes per instance, except for the instance with graph 4-Insertions\_4,  where the algorithm was terminated after 1 hour.

\section{Conclusion} \label{Sec: Conclusion}
We considered the M$k$CS problem and studied an SDP relaxation that was derived from a BSDP formulation of the problem.  
Among other results, we proved that the optimal solution of the SDP relaxation \eqref{SDP}, denoted by $\theta_k(\mc{G})$, is sandwiched between $\alpha_k(\mc{G})$ and the $k$-clique cover number $\Psi_k(\overline{\mc{G}})$, implying that the M$k$CS problem is solvable in polynomial time up to fixed precision for $k$-perfect graphs. 
Moreover, we proposed two novel families of valid inequalities, namely generalized BQP inequalities and generalized rank inequalities. 
We further considered computationally tractable special cases of these inequalities, which we used to strengthen the SDP relaxation.
We efficiently solved the strengthened relaxation using the CP-ADMM. 
We also introduced the INT-ADMM, which yields feasible solutions for the BSDP formulation of the problem. 
The INT-ADMM was obtained from the standard ADMM by adding a third matrix variable that satisfies a sphere constraint. 
While this introduces nonconvexity, a gradual increase of the penalty parameter $\beta$, combined with a mechanism to escape local optima, led in practice to high-quality feasible solutions.

Our computational experiments showed that our valid upper and lower bounds outperform existing bounds on most benchmark instances.
Furthermore, we compared the effectiveness of our valid inequalities, showing in particular that while the generalized triangle inequalities are effective on their own, they have little added value on top of the generalized clique and odd-hole inequalities. 
We also found that limiting the number of new cuts per variable considerably improves the performance of the CP-ADMM. 
In addition, our valid upper bounds for the M$k$CS problem were shown to yield competitive lower bounds for the graph coloring problem.

This research can be extended in several directions.
For example, it would be interesting to incorporate the CP-ADMM and INT-ADMM in a branch-and-bound framework to obtain optimal solutions.
Moreover, the proposed generalized BQP inequalities may be useful to strengthen matrix-lifting relaxations of other packing or partitioning problems, while the generalized rank inequalities could be used for problems involving conflicts, such as the multiple knapsack problem with conflicts, the bin packing problem with conflicts, or quadratic variants thereof.
Finally, we expect that the INT-ADMM can be successfully applied for other optimization problems that can be formulated as a BSDP problem, especially when the SDP relaxation obtained by relaxing integrality can be solved efficiently by the STD-ADMM or CP-ADMM. 

\bibliographystyle{apalike}
\bibliography{references}

@article{Wu2019LpBoxADMM,
  author={Wu, Baoyuan and Ghanem, Bernard},
  journal={IEEE Transactions on Pattern Analysis and Machine Intelligence}, 
  title={$\ell_p$-Box {ADMM}: A Versatile Framework for Integer Programming}, 
  year={2019},
  volume={41},
  number={7},
  pages={1695-1708},
  doi={10.1109/TPAMI.2018.2845842}
}

@article{Oliveira2018ADMMForQAP,
    author = {Danilo Elias Oliveira and Henry Wolkowicz and Yangyang Xu},
    year = {2018},
    title = {{ADMM} for the {SDP} relaxation of the {QAP}},
    journal = {Mathematical Programming Computation},
    volume = {10},
    issue = {4},
    pages = {631-658},
    url = {https://doi.org/10.1007/s12532-018-0148-3}
}

@article{He2016ConvergenceSymmetricADMM,
    author = {He, Bingsheng and Ma, Feng and Yuan, Xiaoming},
    title = {Convergence Study on the Symmetric Version of {ADMM} with Larger Step Sizes},
    journal = {SIAM Journal on Imaging Sciences},
    volume = {9},
    number = {3},
    pages = {1467-1501},
    year = {2016},
    doi = {10.1137/15M1044448},
    URL = {https://doi.org/10.1137/15M1044448}
}

@article{Wu2018LpBoxUsage,
  author={Wu, Qiong and Zhang, Fan and Wang, Hao and Lin, Jun and Liu, Yang},
  journal={IEEE Communications Letters}, 
  title={Parameter-Free  $\ell_p$-Box Decoding of {LDPC} Codes}, 
  year={2018},
  volume={22},
  number={7},
  pages={1318-1321},
  doi={10.1109/LCOMM.2018.2830787}
}

@article{Jiao2021PenalizedLpBox,
  author={Jiao, Xiaopeng and Liu, Haiyang and Mu, Jianjun and He, Yu-Cheng},
  journal={IEEE Transactions on Vehicular Technology}, 
  title={$l_2$-Box {ADMM} Decoding for {LDPC} Codes Over {ISI} Channels}, 
  year={2021},
  volume={70},
  number={4},
  pages={3966-3971},
  doi={10.1109/TVT.2021.3068398}}

@article{deMeijer2024BSDP,
    title={On integrality in semidefinite programming for discrete optimization},
    author={Frank {de Meijer} and Renata Sotirov},
    journal = {SIAM Journal on Optimization},
    volume = {34},
    number = {1},
    pages = {1071-1096},
    year={2024}
}

@misc{Pucher2025PracticalExperience,
    author = {Dunja Pucher and Franz Rendl},
    year = {2025},
    title = {Practical Experience with Stable Set and Coloring Relaxations},
    note = {Peprint, \url{https://arxiv.org/abs/2401.17069}},
    url={https://arxiv.org/abs/2401.17069}
}

@article{Pucher2023SSAndColoringBounds,
    author = {Dunja Pucher and Franz Rendl},
    year = {2023},
    title = {Stable-set and coloring bounds based on 0-1 quadratic optimization},
    journal = {Applied Set-Valued Analysis and Optimization},
    volume = {5},
    number = {2},
    pages = {233-251}
}

@article{Nemhauser1974PropertiesPolyhedra,
    author = {Nemhauser, G. L. and  Trotter, L. E.},
    title = {Properties of vertex packing and independence system polyhedra},
    journal = {Mathematical Programming},
    volume = {6},
    number = {1},
    pages = {48-61},
    year = {1974},
    doi = {10.1007/BF01580222},
    URL = {https://doi.org/10.1007/BF01580222}
}

@article{Letchford2020CliqueAndNodalInequalities,
    title = {The stable set problem: Clique and nodal inequalities revisited},
    journal = {Computers \& Operations Research},
    volume = {123},
    pages = {105024},
    year = {2020},
    issn = {0305-0548},
    doi = {https://doi.org/10.1016/j.cor.2020.105024},
    url = {https://www.sciencedirect.com/science/article/pii/S0305054820301416},
    author = {Adam N. Letchford and Fabrizio Rossi and Stefano Smriglio}
}

@article{Padberg1989BQP,
    author = {Padberg, Manfred},
    title = {The {B}oolean quadric polytope: Some characteristics, facets and relatives},
    year = {1989},
    journal = {Mathematical Programming},
    volume = {45},
    number = {1},
    pages = {139-172},
    doi = {10.1007/BF0158910},
    url = {https://doi.org/10.1007/BF01589101}
}

@article{Januschowski2011MkCS,
    title = {The maximum $k$-colorable subgraph problem and orbitopes},
    journal = {Discrete Optimization},
    volume = {8},
    number = {3},
    pages = {478-494},
    year = {2011},
    issn = {1572-5286},
    doi = {https://doi.org/10.1016/j.disopt.2011.04.002},
    url = {https://www.sciencedirect.com/science/article/pii/S1572528611000272},
    author = {Tim Januschowski and Marc E. Pfetsch},
}

@article{Kuryatnikova2022MkCSProblem,
    author = {Kuryatnikova, Olga and Sotirov, Renata and Vera, Juan C.},
    title = {The Maximum $k$-Colorable Subgraph Problem and Related Problems},
    journal = {INFORMS Journal on Computing},
    volume = {34},
    number = {1},
    pages = {656-669},
    year = {2022},
    doi = {10.1287/ijoc.2021.1086},
    URL = {https://doi.org/10.1287/ijoc.2021.1086},
    eprint = {https://doi.org/10.1287/ijoc.2021.1086}
}

@article{deMeijer2021QCCP,
    author = {Frank {de Meijer} and Renata Sotirov},
    title = {{SDP}-Based Bounds for the Quadratic Cycle Cover Problem via Cutting-Plane Augmented {L}agrangian Methods and Reinforcement Learning},
    journal = {INFORMS Journal on Computing},
    volume = {33},
    number = {4},
    pages = {1262-1276},
    year = {2021},
    doi = {10.1287/ijoc.2021.1075},
    URL = {https://doi.org/10.1287/ijoc.2021.1075},
    eprint = {https://doi.org/10.1287/ijoc.2021.1075}
}

@book{Marshall2011Majorization,
  title={Inequalities: Theory of majorization and its applications},
  author={Marshall, Albert W and Olkin, Ingram and Arnold, Barry C},
  year={2011},
  publisher={Springer, New York, NY}
}

@book{Helmberg2000SDPForCO,
  title={Semidefinite Programming for Combinatorial Optimization},
  author={C. Helmberg},
  year={2000},
  publisher={Konrad-Zuse-Zentrum für Informationstechnik},
  address={Berlin}
}

@article{Gruber2003SSRelaxations,
    author = {Gruber, Gerald and Rendl, Franz},
    title = {Computational Experience with Stable Set Relaxations},
    journal = {SIAM Journal on Optimization},
    volume = {13},
    number = {4},
    pages = {1014-1028},
    year = {2003},
    doi = {10.1137/S1052623401394092},
    URL = {https://doi.org/10.1137/S1052623401394092},
    eprint = {https://doi.org/10.1137/S1052623401394092}
}

@article{Hu2020ADMMQSPP,
    author = {Hu, Hao and Sotirov, Renata},
    title = {On Solving the Quadratic Shortest Path Problem},
    journal = {INFORMS Journal on Computing},
    volume = {32},
    number = {2},
    pages = {219-233},
    year = {2020},
    doi = {10.1287/ijoc.2018.0861},
    URL = {https://doi.org/10.1287/ijoc.2018.0861},
    eprint = {https://doi.org/10.1287/ijoc.2018.0861}
}

@article{deMeijer2023ADMMPartitioning,
    title = {Partitioning through projections: Strong {SDP} bounds for large graph partition problems},
    journal = {Computers \& Operations Research},
    volume = {151},
    pages = {106088},
    year = {2023},
    issn = {0305-0548},
    doi = {https://doi.org/10.1016/j.cor.2022.106088},
    url = {https://www.sciencedirect.com/science/article/pii/S0305054822003185},
    author = {Frank {de Meijer} and Renata Sotirov and Angelika Wiegele and Shudian Zhao}
}

@misc{Sinjorgo2025ADMMStability,
      title={{SDP} bounds on the stability number via {ADMM} and intermediate levels of the {L}asserre hierarchy}, 
      author={Lennart Sinjorgo and Renata Sotirov and Juan C. Vera},
      year={2025},
      note = {Preprint, \url{https://arxiv.org/abs/2506.08648}},
      eprint={2506.08648},
      archivePrefix={arXiv},
      primaryClass={math.OC},
      url={https://arxiv.org/abs/2506.08648}, 
}

@article{deMeijer2025ADMMQMSTP,
    title = {Spanning and splitting: Integer semidefinite programming for the quadratic minimum spanning tree problem},
    journal = {European Journal of Operational Research},
    year = {2025},
    volume = {331},
    number = {2},
    pages = {381-395},
    issn = {0377-2217},
    doi = {https://doi.org/10.1016/j.ejor.2025.10.051},
    url = {https://www.sciencedirect.com/science/article/pii/S0377221725008938},
    author = {Frank {de Meijer} and Melanie Siebenhofer and Renata Sotirov and Angelika Wiegele}
}

@article{Letchford2012BinarySDPMatrices,
    author = {Adam N. Letchford and Michael Sørensen},
    year = {2012},
    title = {Binary positive semidefinite matrices and associated integer polytopes},
    journal = {Mathematical Programming},
    pages = {253-271},
    volume = {131},
    issue = {1},
    url = {https://doi.org/10.1007/s10107-010-0352-z}
}

@article{Li2021PRSMMinCut,
    author = {Li, Xinxin and Pong, Ting Kei and Sun, Hao and Wolkowicz, Henry},
    year = {2021},
    title = {A strictly contractive {P}eaceman-{R}achford splitting method for the doubly nonnegative relaxation of the minimum cut problem},
    journal = {Computational Optimization and Applications},
    pages = {853-891},
    volume = {78},
    issue = {3},
    url = {https://doi.org/10.1007/s10589-020-00261-4}
}

@InProceedings{Boyle1986Dykstra,
    author="Boyle, James P.
    and Dykstra, Richard L.",
    editor="Dykstra, Richard
    and Robertson, Tim
    and Wright, Farroll T.",
    title="A method for finding projections onto the intersection of convex sets in {H}ilbert spaces",
    booktitle="Advances in Order Restricted Statistical Inference",
    year="1986",
    publisher="Springer, New York, NY",
    pages="28--47",
    isbn="978-1-4613-9940-7"
}

@article{Boyd2011ADMMTheory,
    author = {Boyd, Stephen and Parikh, Neal and Chu, Eric and Peleato, Borja and Eckstein, Jonathan},
    year = {2011},
    pages = {1-122},
    title = {Distributed Optimization and Statistical Learning via the Alternating Direction Method of Multipliers},
    volume = {3},
    journal = {Foundations and Trends in Machine Learning},
    doi = {10.1561/2200000016}
}

@phdthesis{Narasimhan1989ThesisMkCSProblem,
  author = {Narasimhan, Giri},
  title  = {The maximum $k$-colorable subgraph problem},
  school = {University of Wisconsin-Madison},
  year   = {1989}
}

@article{Bron1973MaxCliques,
    author = {Bron, Coen and Kerbosch, Joep},
    title = {Algorithm 457: Finding all cliques of an undirected graph},
    year = {1973},
    issue_date = {Sept. 1973},
    publisher = {Association for Computing Machinery},
    address = {New York, NY},
    volume = {16},
    number = {9},
    issn = {0001-0782},
    url = {https://doi.org/10.1145/362342.362367},
    doi = {10.1145/362342.362367},
    journal = {Communications of the ACM},
    month = sep,
    pages = {575–577},
    numpages = {3}
}

@article{Alizadeh1995SDPGenTheta,
    author={Alizadeh,Farid},
    year={1995},
    month={02},
    title={Interior Point Methods in Semidefinite Programming with Applications to Combinatorial Optimization},
    journal={SIAM Journal on Optimization},
    volume={5},
    number={1},
    pages={13-39},
    isbn={10526234},
    language={English},
    url={https://tilburguniversity.idm.oclc.org/login?url=https://www.proquest.com/scholarly-journals/interior-point-methods-semidefinite-programming/docview/920195812/se-2},
}

@InProceedings{Januschowski2011ILPSym,
    author="Januschowski, Tim
    and Pfetsch, Marc E.",
    editor="Achterberg, Tobias
    and Beck, J. Christopher",
    title="Branch-Cut-and-Propagate for the Maximum k-Colorable Subgraph Problem with Symmetry",
    booktitle="Integration of AI and OR Techniques in Constraint Programming for Combinatorial Optimization Problems",
    year="2011",
    publisher="Springer, Berlin,  Heidelberg",
    pages="99--116",
    isbn="978-3-642-21311-3"
}

@article{Campelo2010ILPLagr,
    title = {A Combined Parallel {L}agrangian Decomposition and Cutting-Plane Generation for Maximum Stable Set Problems},
    journal = {Electronic Notes in Discrete Mathematics},
    volume = {36},
    pages = {503-510},
    year = {2010},
    note = {ISCO 2010 - International Symposium on Combinatorial Optimization},
    issn = {1571-0653},
    doi = {https://doi.org/10.1016/j.endm.2010.05.064},
    url = {https://www.sciencedirect.com/science/article/pii/S157106531000065X},
    author = {Manoel Campêlo and Ricardo C. Corrêa}
}

@article{Quintero2022Quantum,
    title = {Characterization of {QUBO} reformulations for the maximum $k$-colorable subgraph problem},
    journal = {Quantum Information Processing},
    volume = {21},
    pages = {89},
    year = {2022},
    doi = {10.1007/s11128-022-03421-z},
    url = {https://doi.org/10.1007/s11128-022-03421-z},
    author = {Quintero, Rodolfo and Bernal, David and Terlaky, Tamás and Zuluaga, Luis F.}
}

@article{Narasimhan1990GenTheta,
    title = {A generalization of {L}ovász $\vartheta$ function},
    author = {G. Narasimhan and R. Manber},
    journal = {DIMACS Series in Discrete Mathematics and Theoretical Computer Science},
    volume = {1},
    pages = {19-27},
    number = {},
    year = {1990},
    doi = {},
}

@article{Lewis1980NPHardness,
title = {The node-deletion problem for hereditary properties is {NP}-complete},
journal = {Journal of Computer and System Sciences},
volume = {20},
number = {2},
pages = {219-230},
year = {1980},
issn = {0022-0000},
doi = {https://doi.org/10.1016/0022-0000(80)90060-4},
url = {https://www.sciencedirect.com/science/article/pii/0022000080900604},
author = {John M. Lewis and Mihalis Yannakakis}
}

@article{ApplicationChannelAssignment1,
author = {Halldórsson, Magnús and Li, Li and Joseph Halpern and Mirrokni, Vahab},
year = {2004},
month = {07},
pages = {235-248},
title = {On spectrum sharing games},
volume = {22},
number = {4},
journal = {Distributed Computing},
doi = {10.1007/s00446-010-0098-0}
}

@INPROCEEDINGS{ApplicationChannelAssignment2,
  author={Subramanian, Anand Prabhu and Gupta, Himanshu and Das, Samir R. and Buddhikot, Milind M.},
  booktitle={2007 2nd IEEE International Symposium on New Frontiers in Dynamic Spectrum Access Networks}, 
  title={Fast Spectrum Allocation in Coordinated Dynamic Spectrum Access Based Cellular Networks}, 
  year={2007},
  volume={},
  number={},
  pages={320-330},
  doi={10.1109/DYSPAN.2007.50}
  }

@Inproceedings{ApplicationChannelAssignment3,
    author = {A.M.C.A. Koster and M. Scheffel},
    title = {A routing and network dimensioning strategy to reduce wavelength continuity conflicts in all-optical networks.},
    booktitle = {Proceedings of the International Network Optimization Conference (INOC) 2007}, 
    year = {2007}
}

@article{ApplicationChannelAssignment4,
    author = {Hertz, A. and Montagné, R. and Gagnon, F.},
    year = {2016},
    title = {Constructive algorithms for the partial directed weighted improper coloring problem},
    journal = {Journal of Graph Algorithms and Applications},
    volume = {20},
    number = {2}, 
    pages = {159–188},
    url = {https://doi.org/10.7155/jgaa.00389}
}

@article{ApplicationChannelAssignment5,
title = {Inductive $k$-independent graphs and $c$-colorable subgraphs in scheduling: A review},
author = {Bentert, Matthias and Bevern, René and Niedermeier, Rolf},
year = {2019},
journal = {Journal of Scheduling},
volume = {22},
number = {1},
pages = {3-20},
url = {https://EconPapers.repec.org/RePEc:spr:jsched:v:22:y:2019:i:1:d:10.1007_s10951-018-0595-8}
}

@article{ApplicationVLSI1,
  author={Marek-Sadowska, M.},
  journal={IEEE Transactions on Computer-Aided Design of Integrated Circuits and Systems}, 
  title={An Unconstrained Topological Via Minimization Problem for Two-Layer Routing}, 
  year={1984},
  volume={3},
  number={3},
  pages={184-190},
  doi={10.1109/TCAD.1984.1270074}
}

@article{ApplicationVLSI2,
    title = {Solving {VLSI} design and {DNA} sequencing problems using bipartization of graphs},
    journal = {Computational Optimization and Applications},
    volume = {51},
    number = {2},
    pages = {749-781},
    year = {2012},
    doi = {10.1007/s10589-010-9355-1},
    author = {Fouilhoux, Pierre and Mahjoub, A. Ridha}
}

@article{ApplicationGenetics,
  title={Algorithmic strategies for the single nucleotide polymorphism haplotype assembly problem},
  author={Ross Lippert and Russell Schwartz and Giuseppe Lancia and Sorin Istrail},
  journal={Briefings in bioinformatics},
  year={2002},
  volume={3},
  number={1},
  pages={23-31},
  url={https://api.semanticscholar.org/CorpusID:14950824}
}

@article{ApplicationJobScheduling,
    title = {On the $k$-coloring of intervals},
    journal = {Discrete Applied Mathematics},
    volume = {59},
    number = {3},
    pages = {225-235},
    year = {1995},
    issn = {0166-218X},
    doi = {https://doi.org/10.1016/0166-218X(95)80003-M},
    url = {https://www.sciencedirect.com/science/article/pii/0166218X9580003M},
    author = {Martin C. Carlisle and Errol L. Lloyd}
}

@article{Lovasz1979,
    author = {L\'aszl\'o Lov\'asz},
    year = {1979},
    month = {02},
    pages = {1-7},
    title = {On the {S}hannon capacity of a graph},
    volume = {25},
    journal = {Information Theory, IEEE Transactions on},
    doi = {10.1109/TIT.1979.1055985}
}

@article{Schrijver1979,
  author={Schrijver, A.},
  journal={IEEE Transactions on Information Theory}, 
  title={A comparison of the {D}elsarte and {L}ovász bounds}, 
  year={1979},
  volume={25},
  number={4},
  pages={425-429},
  doi={10.1109/TIT.1979.1056072}
}

@article{Sinjorgo2022GenVartheta,
    author = {Sinjorgo, Lennart and Sotirov, Renata},
    title = {On the Generalized $\vartheta$-Number and Related Problems for Highly Symmetric Graphs},
    journal = {SIAM Journal on Optimization},
    volume = {32},
    number = {2},
    pages = {1344-1378},
    year = {2022},
    doi = {10.1137/21M1414620},
    URL = {https://doi.org/10.1137/21M1414620},
    eprint = {https://doi.org/10.1137/21M1414620}
}

@book{Johnson1996DIMACS,
    author = {David S. Johnson and Michael A. Trick},
    title = {Cliques, Coloring, and Satisfiability: Second DIMACS Implementation Challenge, Workshop, October 11-13, 1993},
    year = {1996},
    isbn = {0821866095},
    publisher = {AMS, Boston, MA},
}

@misc{Johnson2002COLOR02,
    author = {D. S. Johnson and A. Mehrotra and M. A. Trick},
    title = {{COLOR}02/03/04: Graph coloring and its generalizations},
    year = {2002},
    howpublished  = {\url{https://mat.tepper.cmu.edu/COLOR04}}
}

@misc{CVX,
  author       = {Michael Grant and Stephen Boyd},
  title        = {{CVX}: Matlab Software for Disciplined Convex Programming, version 2.1},
  howpublished = {\url{https://cvxr.com/cvx}},
  month        = mar,
  year         = 2014
}

@article{Dukanovic2007SDPGraphColoring,
    author = {Dukanovic, Igor and Rendl, Franz},
    title = {Semidefinite programming relaxations for graph coloring and maximal clique problems},
    year = {2007},
    issue_date = {March 2007},
    publisher = {Springer-Verlag},
    address = {Berlin, Heidelberg},
    volume = {109},
    issn = {0025-5610},
    journal = {Mathematical Programming},
    month = mar,
    pages = {345–365},
    numpages = {21}
}

@article{Gaar2020SDPGraphColoring,
    author = {Gaar, Elisabeth and Rendl, Franz},
    year = {2020},
    title = {A computational study of exact subgraph based {SDP} bounds for Max-Cut, stable set and coloring},
    journal = {Mathematical Programming},
    pages = {283-308},
    volume = {183},
    issue = {1},
    url = {https://doi.org/10.1007/s10107-020-01512-2}
}

@article{Greene1976,
  author  = {C. Greene},
  title   = {Some partitions associated with a partially ordered set},
  journal = {Journal of Combinatorial Theory, Series A},
  volume  = {20},
  number  = {1},
  pages   = {69--79},
  year    = {1976}
}

@article{GreeneKleitman1976,
  author  = {C. Greene and D. J. Kleitman},
  title   = {The structure of {S}perner $k$-families},
  journal = {Journal of Combinatorial Theory, Series A},
  volume  = {20},
  number  = {1},
  pages   = {41--68},
  year    = {1976}
}

@incollection{lovasz1983perfect,
  author    = {L{\'a}szl{\'o} Lov{\'a}sz},
  title     = {Perfect Graphs},
  booktitle = {Selected Topics in Graph Theory II},
  editor    = {L. W. Beineke and R. J. Wilson},
  publisher = {Academic Press},
  year      = {1983},
  pages     = {55--87}
}

@article{Berge1992,
  author  = {C. Berge},
  title   = {The $q$-perfect graphs, {II}},
  journal = {Matematiche (Catania)},
  volume  = {47},
  year    = {1992},
  pages   = {205--211}
}

@mastersthesis{Oliveira2024,
  author       = {Oliveira, T. L.},
  title        = {The maximum $k$-colorable subgraph problem},
  school       = {Universidade de S\~ao Paulo},
  type         = {Master thesis},
  year         = {2024},
  doi          = {10.11606/D.45.2024.tde-25082025-111037},
  url          = {https://www.teses.usp.br},
}

@article{Grtschel1984PolynomialAF,
  title={Polynomial Algorithms for Perfect Graphs},
  author={Martin Gr{\"o}tschel and L{\'a}szl{\'o} Lov{\'a}sz and Alexander Schrijver},
  journal={Annals of discrete mathematics},
  year={1984},
  volume={88},
  pages={325-356},
  url={https://api.semanticscholar.org/CorpusID:3871264}
}

@article{Chudnovsky2006SPGT,
  author  = {Maria Chudnovsky and Neil Robertson and Paul Seymour and Robin Thomas},
  title   = {The Strong Perfect Graph Theorem},
  journal = {Annals of Mathematics},
  volume  = {164},
  number  = {1},
  pages   = {51--229},
  year    = {2006}
}

\section*{Statements and declarations}
\paragraph{Funding} The authors declare that no funds, grants, or other support were received during the preparation of this manuscript.
\paragraph{Competing interests} The authors have no relevant financial or non-financial interests to disclose.
\paragraph{Code availability} Our code is available upon request.
\paragraph{Data availability} The instances considered in our computational experiments can be found at 
\url{https://github.com/jamestrimble/max-weight-clique-instances/tree/master/DIMACS/weighted} (graphs from the DIMACS Implementation Challenge),
\url{https://mat.tepper.cmu.edu/COLOR04/} (graphs from the COLOR02 symposium), and 
\url{https://arxiv.org/abs/2401.17069} (random graphs from \citep{Pucher2025PracticalExperience}).

\end{document}